\documentclass[a4paper]{article}
\usepackage{amssymb,amsmath,amsthm}
\usepackage{epsfig}
\usepackage{hyperref}
\newcommand{\ie}{\emph{i.e.}}
\newcommand{\eg}{\emph{e.g.}}
\newcommand{\cf}{\emph{cf}}
\newcommand{\Nat}{\mathbb{N}}
\newcommand{\Int}{\mathbb{Z}}
\newcommand{\Real}{\mathbb{R}}

\newcommand{\sii}{L^2}
\newcommand{\si}{L^1}
\newcommand{\Dom}{\mathfrak{D}}

\newcommand{\supp}{\mathop{\mathrm{supp}}\nolimits}
\newcommand{\Hilbert}{\mathcal{H}}
\newcommand{\eps}{\varepsilon}
\numberwithin{equation}{section}
\newtheorem{Lemma}{Lemma}[section]
\newtheorem{Theorem}{Theorem}[section]
\newtheorem{Proposition}{Proposition}[section]
\newtheorem{Corollary}{Corollary}[section]
\newtheorem*{Conjecture}{Conjecture}
\theoremstyle{remark}
\newtheorem{Remark}{Remark}[section]
\theoremstyle{definition}
\newtheorem{Definition}{Definition}[section]
\begin{document}
%
\title{\textbf{\Large
The Hardy inequality and the heat equation in twisted tubes
}}
\author{
David Krej\v{c}i\v{r}{\'\i}k$^a$ \ and \ Enrique Zuazua$^b$
}
\date{
\ \vspace{-5ex} \\
\small
\emph{
\begin{quote}
\begin{itemize}
\item[$a)$]
Department of Theoretical Physics,
Nuclear Physics Institute,
A\-cad\-e\-my of Sciences, 25068 \v Re\v z, Czech Republic;
krejcirik@ujf.cas.cz
\item[$b)$]
Basque Center for Applied Mathematics,
Bizkaia Technology Park, Building 500,
48160 Derio, Basque Country, Spain;
zuazua@bcamath.org
\end{itemize}
\end{quote}
}
\ \vspace{-4ex} \\
\begin{center}
22 June 2009
\end{center}
\ \vspace{-7ex} \\
}
\maketitle
\begin{abstract}
We show that a twist of a three-dimensional tube of uniform cross-section
yields an improved decay rate for the heat semigroup associated
with the Dirichlet Laplacian in the tube.
The proof employs Hardy inequalities
for the Dirichlet Laplacian in twisted tubes
and the method of self-similar variables
and weighted Sobolev spaces for the heat equation.
\end{abstract}
%
{\scriptsize
\tableofcontents
}

%
\newpage
\section{Introduction}
%
It has been shown recently in~\cite{EKK} that a local twist of
a straight three-di\-men\-sional tube $\Omega_0:=\Real\times\omega$
of non-circular cross-section $\omega\subset\Real^2$
leads to an effective repulsive interaction
in the Schr\"odinger equation of a quantum particle
constrained to the twisted tube~$\Omega_\theta$.
More precisely, there is a Hardy-type inequality
for the particle Hamiltonian modelled
by the Dirichlet Laplacian~$-\Delta_D^{\Omega_\theta}$
at its threshold energy~$E_1$
if, and only if, the tube is twisted
(\cf~Figure~\ref{Fig.3D_twist}).
That is, the inequality
\begin{equation}\label{I.Hardy}
  -\Delta_D^{\Omega_\theta} - E_1 \geq \varrho
\end{equation}
holds true, in the sense of quadratic forms in $\sii(\Omega_\theta)$,
with a positive function~$\varrho$ provided that the tube is twisted,
while~$\varrho$ is necessarily zero for~$\Omega_0$.
Here~$E_1$ coincides with the first eigenvalue
of the Dirichlet Laplacian~$-\Delta_D^{\omega}$
in the cross-section~$\omega$.

\begin{figure}[h!]
\begin{center}
\epsfig{file=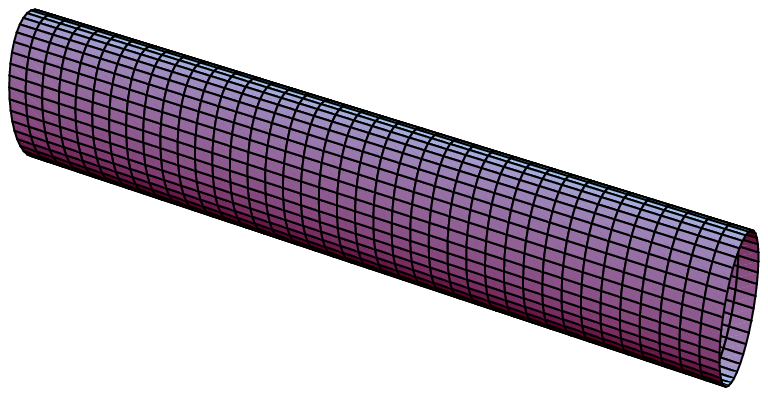,width=0.45\textwidth}
\epsfig{file=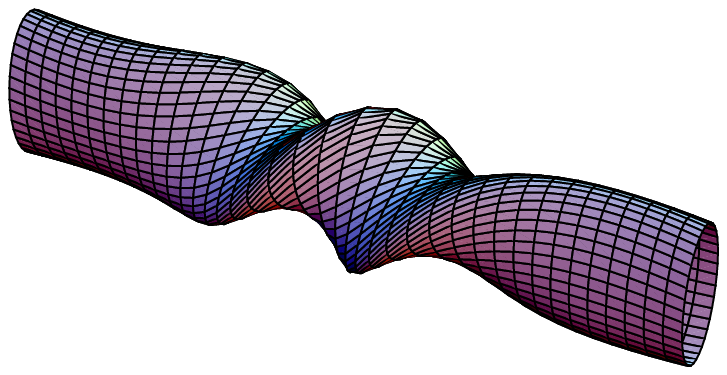,width=0.49\textwidth}
\end{center}
\caption{
Untwisted and twisted tubes of elliptical cross-section.
}
\label{Fig.3D_twist}
\end{figure}

The inequality~\eqref{I.Hardy} has important consequences
for conductance properties of quantum waveguides.
It clearly implies the absence of bound states
(\ie, stationary solutions to the Schr\"odinger equation)
below the energy~$E_1$ even if the particle
is subjected to a small attractive interaction,
which can be either of potential or geometric origin
(\cf~\cite{EKK} for more details).
At the same time, a repulsive effect of twisting
on eigenvalues embedded in the essential spectrum
has been demonstrated in~\cite{Kovarik-Sacchetti_2007}.
Hence, roughly speaking, the twist prevents the particle
to be trapped in the waveguide.
Additional spectral properties of twisted tubes have been studied in
\cite{EKov_2005,K6-erratum,Briet-Kovarik-Raikov-Soccorsi_2008}.

It is natural to ask whether the repulsive
effect of twisting demonstrated in~\cite{EKK} in the quantum context
has its counterpart in other areas of physics, too.
The present paper gives an affirmative answer to this question
for systems modelled by the diffusion equation in the tube~$\Omega_\theta$:
\begin{equation}\label{I.heat}
  u_t - \Delta u = 0
  \,,
\end{equation}
subject to Dirichlet boundary conditions on~$\partial\Omega_\theta$.
Indeed, we show that the twist is responsible
for a faster convergence of the solutions of~\eqref{I.heat}
to the (zero) stable equilibrium.
The second objective of the paper is to give a new (simpler and more direct)
proof of the Hardy inequality~\eqref{I.Hardy}
under weaker conditions than those in~\cite{EKK}.

\subsection{The main result}
%
Before stating the main result about the large time
behaviour of the solutions to~\eqref{I.heat},
let us make some comments on the subtleties
arising with the study of the heat equation in~$\Omega_\theta$.

The specific deformation~$\Omega_\theta$ of~$\Omega_0$
via twisting we consider can be visualized as follows:
instead of simply translating~$\omega$ along~$\Real$
we also allow the (non-circular) cross-section~$\omega$ to rotate
with respect to a (non-constant) angle $x_1\mapsto\theta(x_1)$.
See Figure~\ref{Fig.3D_twist}
(the precise definition is postponed until Section~\ref{Sec.Pre},
\cf~Definition~\ref{definition}).
We assume that the deformation is local, \ie,
\begin{equation}\label{locally}
  \dot\theta
  \
  \mbox{has compact support in $\Real$}.
\end{equation}
Then the straight and twisted tubes have the same spectrum
(\cf~\cite[Sec.~4]{K6}):
\begin{equation}\label{spectrum}
  \sigma(-\Delta_D^{\Omega_\theta})
  = \sigma_\mathrm{ess}(-\Delta_D^{\Omega_\theta})
  = [E_1,\infty)
  \,.
\end{equation}
The fine difference between twisted and untwisted tubes
in the spectral setting is reflected in the existence of~\eqref{I.Hardy}
for the former.

In view of the spectral mapping theorem,
the indifference~\eqref{spectrum} transfers to the following identity
for the heat semigroup:
\begin{equation}\label{E.decay}
  \forall t \geq 0 \,, \qquad
  \big\|
  e^{\Delta_D^{\Omega_\theta} t}
  \big\|_{\sii(\Omega_\theta)\to\sii(\Omega_\theta)}
  = e^{-E_1 t}
  \,,
\end{equation}
irrespectively whether the tube~$\Omega_\theta$ is twisted or not.
That is, we clearly have the exponential decay
\begin{equation}\label{solution.norate}
  \|u(t)\|_{\sii(\Omega_\theta)}
  \leq e^{-E_1 t} \, \|u_0\|_{\sii(\Omega_\theta)}
\end{equation}
for each time $t \geq 0$ and any initial datum~$u_0$ of~\eqref{I.heat}.
To obtain some finer differences as regards the time-decay of solutions,
it is therefore natural to consider rather the \emph{``shifted''} semigroup
\begin{equation}\label{semigroup}
  S(t) := e^{(\Delta_D^{\Omega}+E_1)t}
\end{equation}
as an operator from a \emph{subspace}
of $\sii(\Omega_\theta)$ to $\sii(\Omega_\theta)$.

In this paper we mainly (but not exclusively)
consider the subspace of initial data given by the weighted space
\begin{equation}\label{weight}
  \sii(\Omega_\theta,K)
  \qquad\mbox{with}\qquad
  K(x) := e^{x_1^2/4}
  \,,
\end{equation}
and study the asymptotic properties of the semigroup
via the \emph{decay rate} defined by
\begin{equation*}
  \Gamma(\Omega_\theta) :=
  \sup \Big\{ \Gamma \left| \
  \exists C_\Gamma > 0, \, \forall t \geq 0, \
  \|S(t)\|_{
  \sii(\Omega_\theta,K)
  \to
  \sii(\Omega_\theta)
  }
  \leq C_\Gamma \, (1+t)^{-\Gamma}
  \Big\} \right.
  \!.
\end{equation*}
Our main result reads as follows:
\begin{Theorem}\label{Thm.rate}
Let $\theta\in C^1(\Real)$ satisfy~\eqref{locally}.
We have
$$
  \Gamma(\Omega_\theta)
  \begin{cases}
    \, = 1/4
    & \mbox{if $\Omega_\theta$ is untwisted},
    \\
    \, \geq 3/4
    & \mbox{if $\Omega_\theta$ is twisted}.
  \end{cases}
$$
\end{Theorem}

The statement of the theorem for solutions~$u$ of~\eqref{I.heat}
in~$\Omega_\theta$ can be reformulated as follows.
For every $\Gamma < \Gamma(\Omega_\theta)$, there exists
a positive constant~$C_\Gamma$ such that
\begin{equation}\label{solution.rate}
  \|u(t)\|_{\sii(\Omega_\theta)}
  \leq C_\Gamma \, (1+t)^{-\Gamma} \, e^{-E_1 t} \,
  \|u_0\|_{\sii(\Omega_\theta,K)}
\end{equation}
for each time $t \geq 0$ and any initial datum $u_0 \in \sii(\Omega_\theta,K)$.
This should be compared with the inequality~\eqref{solution.norate}
which is sharp in the sense that it does not allow
for any extra polynomial-type decay rate due to~\eqref{E.decay}.
On the other hand, we see that the decay rate
is at least three times better in a twisted tube
provided that the initial data are restricted to the weighted space.

A type of the estimate~\eqref{solution.rate} in an untwisted tube
can be obtained in a less restrictive weighted space
(\cf~Theorem~\ref{Thm.decay.1D}).
The power~$1/4$ actually reflects the quasi-one-dimensional
nature of our model.
Indeed, in the whole Euclidean space
one has the well known dimensional bound
\begin{equation}\label{P.decay}
  \forall t \geq 0 \,, \qquad
  \big\|e^{\Delta_D^{\Real^d} t}\big\|_{
  \sii(\Real^d,K)
  \to
  \sii(\Real^d)
  }
  \, \leq \, (1+t)^{-d/4}
  \,.
\end{equation}
The fact that the power~$1/4$ is optimal for untwisted tubes
can be established quite easily by a ``separation of variables''
(\cf~Proposition~\ref{Prop.decay.1D}).
The fine effect of twisting is then reflected in the positivity
of $\Gamma(\Omega_\theta)-1/4$;
in view of~\eqref{P.decay}, it can be interpreted as
``enlarging the dimension'' of the tube.

\subsection{The idea of the proof}
%
The principal idea behind the main result of Theorem~\ref{Thm.rate},
\ie\ the better decay rate in twisted tubes,
is the positivity of the function~$\varrho$ in~\eqref{I.Hardy}.
In fact, Hardy inequalities have already been used
as an essential tool to study the asymptotic behaviour of the heat equation
in other situations~\cite{Cabre-Martel_1999,Vazquez-Zuazua_2000}.
However, it should be stressed that Theorem~\ref{Thm.rate} does not follow
as a direct consequence of~\eqref{I.Hardy} by some energy estimates
(\cf~Section~\ref{Sec.failure}) but that important and further
technical developments that we explain now are needed.
Nevertheless, overall, the main result of the paper confirms that the Hardy
inequalities end up enhancing the decay rate
of solutions..

Let us now briefly describe our proof (as given in Section~\ref{Sec.similar})
that there is the extra decay rate if the tube is twisted.
\smallskip \\
\textbf{I.}
First, we map the twisted tube~$\Omega_\theta$
to the straight one~$\Omega_0$ by a change of variables,
and consider rather the transformed (and shifted by~$E_1$) equation
\begin{equation}\label{heat.straight}
  u_t
  - (\partial_1-\dot{\theta}\,\partial_\tau)^2u
  - \Delta' u - E_1 u = 0
\end{equation}
in~$\Omega_0$ instead of~\eqref{I.heat}.
Here $-\Delta' := - \partial_2^2 - \partial_3^2$
and $\partial_\tau := x_3 \partial_2 - x_2 \partial_3$,
with $x=(x_1,x_2,x_3) \in \Omega_0$, denote the ``transverse'' Laplace
and angular-derivative operators, respectively.%
\smallskip \\
\textbf{II.}
The main ingredient in the subsequent analysis
is the method of self-similar solutions
developed in the whole Euclidean space
by Escobedo and Kavian~\cite{Escobedo-Kavian_1987}.
Writing
\begin{equation}\label{SST}
  \tilde{u}(y_1,y_2,y_3,s)
  = e^{s/4} u(e^{s/2}y_1,y_2,y_3,e^s-1)
  \,,
\end{equation}
the equation~\eqref{heat.straight} is transformed to
\begin{equation}\label{heat.similar}
  \tilde{u}_s
  - \mbox{$\frac{1}{2}$} \, y_1 \;\! \partial_1 \tilde{u}
  - (\partial_1-\sigma_s\,\partial_\tau)^2 \tilde{u}
  - e^s \, \Delta' \tilde{u}
  - E_1 \, e^s \, \tilde{u}
  - \mbox{$\frac{1}{4}$} \, \tilde{u}
  = 0
\end{equation}
in self-similarity variables $(y,s)\in\Omega_0\times(0,\infty)$,
where
\begin{equation}\label{sigma}
  \sigma_s(y_1) := e^{s/2} \dot{\theta}(e^{s/2}y_1)
  \,.
\end{equation}

Note that~\eqref{heat.similar} is a parabolic equation
with \emph{time-dependent} coefficients.
This non-autonomous feature is a consequence of
the non-trivial geometry we deal with
and represents thus the main difficulty in our study.
We note that an analogous difficulty has been encountered
previously for a convection-diffusion equation
in the whole space but with a variable diffusion coefficient
\cite{Duro-Zuazua_1999}.
\smallskip \\
\textbf{III.}
We reconsider~\eqref{heat.similar}
in the weighted space~\eqref{weight} and show
that the associated generator has purely discrete spectrum then.
Now a difference with respect to the self-similarity
transformation in the whole Euclidean space is that
the generator is not a symmetric operator if the tube is twisted.
However, this is not a significant obstacle since only the real
part of the corresponding quadratic form
is relevant for subsequent energy estimates (\cf~\eqref{formal}).
\smallskip \\
\textbf{IV.}
Finally, we look at the asymptotic behaviour of~\eqref{heat.similar}
as the self-similar time~$s$ tends to infinity.
Assume that the tube is twisted.
The scaling coming from the self-similarity transformation
is such that the function~\eqref{sigma} converges
in a distributional sense to a multiple of the delta function
supported at zero as $s\to\infty$.
The square of~$\sigma_s$ becomes therefore extremely singular
at the section $\{0\}\times\omega$ of the tube for large times.
At the same time, the prefactors~$e^s$ in~\eqref{heat.similar} diverge
exactly as if the cross-section of the tube shrunk to zero $s\to\infty$.
Taking these two simultaneous limits into account,
it is expectable that~\eqref{heat.similar}
will be approximated for large times by
the essentially one-dimensional problem
\begin{equation}\label{heat.1D}
  \varphi_s
  - \mbox{$\frac{1}{2}$} \, y_1 \;\! \varphi_{y_1}
  - \varphi_{y_1y_1}
  - \mbox{$\frac{1}{4}$} \, \varphi
  = 0
  \,, \qquad
  s \in (0,\infty), \ y_1 \in \Real \,,
\end{equation}
with an extra Dirichlet boundary condition at $y_1=0$.
This evolution equation is explicitly solvable in $\sii(\Real,K)$
and it is easy to see that
\begin{equation}\label{heat.1D.estimate}
  \|\varphi\|_{\sii(\Real,K)}
  \leq e^{-\frac{3}{4} s} \, \|\varphi_0\|_{\sii(\Real,K)}
  \,,
\end{equation}
for any initial datum~$\varphi_0$.
Here the exponential decay rate transfers to a polynomial one
after returning to the original time~$t$, and the number~$3/4$
gives rise to that of the bound of Theorem~\ref{Thm.rate} in the twisted case.

On the other hand, we get just~$1/4$ in~\eqref{heat.1D.estimate}
provided that the tube is untwisted
(which corresponds to imposing no extra condition at $y_1=0$).

\smallskip
Two comments are in order.
First, we do not establish any theorem
that solutions of~\eqref{heat.similar} can be approximated
by those of~\eqref{heat.1D} as $s \to \infty$.
We only show a strong-resolvent convergence for operators
related to their generators (Proposition~\ref{Prop.strong}).
This is, however, sufficient to prove Theorem~\ref{Thm.rate}
with help of energy estimates.
Proposition~\ref{Prop.strong} is probably the most
significant auxiliary result of the paper
and we believe it is interesting in its own right.

Second, in the proof of Proposition~\ref{Prop.strong}
we essentially use the existence of
the Hardy inequality~\eqref{I.Hardy} in twisted tubes.
In fact, the positivity of~$\varrho$ is directly responsible
for the extra Dirichlet boundary condition of~\eqref{heat.1D}.
Since the Hardy inequality holds
in the Hilbert space $\sii(\Omega_0)$ (no weight),
Proposition~\ref{Prop.strong} is stated for operators
transformed to it from~\eqref{weight} by an obvious unitary transform.
In particular, the asymptotic operator~$h_D$ of Proposition~\ref{Prop.strong}
acts in a different space, $\sii(\Real)$,
but it is unitarily equivalent to the generator of~\eqref{heat.1D}.

\subsection{The content of the paper}
%
The organization of this paper is as follows.

In the following Section~\ref{Sec.Pre}
we give a precise definition of twisted tubes~$\Omega_\theta$
and the corresponding Dirichlet Laplacian $-\Delta_D^{\Omega_\theta}$.

Section~\ref{Sec.Hardy} is mainly devoted to a new proof
of the Hardy inequality (Theorem~\ref{Thm.Hardy})
as announced in~\cite{K6-erratum}.
We mention its consequences on the stability of the spectrum
of the Laplacian (Proposition~\ref{Prop.difference})
and emphasize that the Hardy weight cannot be made
arbitrarily large by increasing the twisting
(Proposition~\ref{Prop.limit}).
Finally, we establish there a new Sobolev-type inequality
in twisted tubes (Theorem~\ref{Thm.Sobolev}).

The heat equation in twisted tubes is considered in Section~\ref{Sec.heat}.
Using some energy-type estimates,
we prove in Theorems~\ref{Thm.decay.1D} and~\ref{Thm.decay}
polynomial-type decay results for the heat semigroup
as a consequence of the Sobolev and Hardy inequalities, respectively.
Unfortunately, Theorem~\ref{Thm.decay} does not represent
any improvement upon the 1/4-decay rate of Theorem~\ref{Thm.decay.1D}
which is valid in untwisted tubes as well.

The main body of the paper is therefore represented
by Section~\ref{Sec.similar}
where we develop the method of self-similar solutions
to get the improved decay rate of Theorem~\ref{Thm.rate}
as described above.
Furthermore, in Section~\ref{Sec.alternative} we establish
an alternative version of Theorem~\ref{Thm.rate}.

The paper is concluded in Section~\ref{Sec.end}
by referring to physical interpretations of the result
and to some open problems.

\section{Preliminaries}\label{Sec.Pre}
%
In this section we introduce some basic definitions
and notations we shall use throughout the paper.

\subsection{The geometry of a twisted tube}
%
Given a bounded open connected set $\omega \subset\Real^2$,
let $\Omega_0:=\Real\times\omega$ be a straight tube
of cross-section~$\omega$.
We assume no regularity hypotheses about~$\omega$.
Let $\theta:\Real\to\Real$ be a $C^1$-smooth function
with bounded derivative
(occasionally we will denote by the same symbol~$\theta$
the function $\theta \otimes 1$ on $\Omega_0$).
We introduce another tube of the same cross-section~$\omega$
as the image
$$
  \Omega_\theta := \mathcal{L}_\theta(\Omega_0)
  \,,
$$
where the mapping $\mathcal{L}_\theta:\Real^3\to\Real^3$
is given by
\begin{equation}\label{diffeomorphism}
  \mathcal{L}_\theta(x)
  := \big(
  x_1,
  x_2\cos\theta(x_1)+x_3\sin\theta(x_1),
  -x_2\sin\theta(x_1)+x_3\cos\theta(x_1)
  \big)
  \,.
\end{equation}
\begin{Definition}[Twisted and untwisted tubes]\label{definition}
We say that the tube~$\Omega_\theta$ is \emph{twisted}
if the following two conditions are satisfied:
\begin{enumerate}
\item
$\theta$ is not constant,
\item
$\omega$ is not rotationally symmetric with respect to
the origin in~$\Real^2$.
\end{enumerate}
Otherwise we say that~$\Omega_\theta$ is \emph{untwisted}.
\end{Definition}

Here the precise meaning of~$\omega$ being
``rotationally symmetric with respect to the origin in~$\Real^2$''
is that, for every $\vartheta\in(0,2\pi)$,
$$
  \omega_\vartheta :=
  \left\{
  x_2\cos\vartheta+x_3\sin\vartheta,
  -x_2\sin\vartheta+x_3\cos\vartheta
  \, \big| \, (x_2,x_3)\in\omega
  \right\}
  = \omega
  \,,
$$
with the natural convention that we identify~$\omega$ and~$\omega_\vartheta$
(and other open sets)
provided that they differ on a set of zero capacity.
Hence, modulus a set of zero capacity,
$\omega$~is rotationally symmetric
with respect to the origin in~$\Real^2$ if, and only if,
it is a disc or an annulus centered at the origin of~$\Real^2$.
In view of this convention, any untwisted $\Omega_\theta$
can be identified with the straight tube~$\Omega_0$
by an isometry of the Euclidean space.

We write $x=(x_1,x_2,x_3)$ for a point/vector in~$\Real^3$.
If~$x$ is used to denote a point in~$\Omega_0$ or~$\Omega_\theta$,
we refer to~$x_1$ and $x':=(x_2,x_3)$ as ``longitudinal''
and ``transverse'' variables in the tube, respectively.

It is easy to check that the mapping~$\mathcal{L}_\theta$ is injective
and that its Jacobian is identically equal to~$1$.
Consequently, $\mathcal{L}_\theta$ induces a (global) diffeomorphism
between~$\Omega_0$ and~$\Omega_\theta$.

\subsection{The Dirichlet Laplacian in a twisted tube}
%
It follows from the last result that~$\Omega_\theta$ is an open set.
The corresponding Dirichlet Laplacian in~$\sii(\Omega_\theta)$
can be therefore introduced in a standard way
as the self-adjoint operator $-\Delta_D^{\Omega_\theta}$
associated with the quadratic form
$$
  Q_D^{\Omega_\theta}[\Psi] := \|\nabla\Psi\|_{\sii(\Omega_\theta)}^2 ,
  \qquad
  \Psi \in \Dom(Q_D^{\Omega_\theta}) := H_0^1(\Omega_\theta)
  \,.
$$
By the representation theorem,
$
  -\Delta_D^{\Omega_\theta}\Psi = -\Delta\Psi
$
for
$
  \Psi \in \Dom(-\Delta_D^{\Omega_\theta}) :=
  \{\Psi\in H_0^1(\Omega_\theta) \, | \, \Delta\Psi \in \sii(\Omega_\theta)\}
$,
where the Laplacian $\Delta\Psi$ should be understood
in the distributional sense.

Moreover, using the diffeomorphism induced by~$\mathcal{L}_\theta$,
we can ``untwist'' the tube by expressing
the Laplacian $-\Delta_D^{\Omega_\theta}$
in the curvilinear coordinates determined by~\eqref{diffeomorphism}.
More precisely, let~$U_\theta$ be the unitary transformation
from $\sii(\Omega_\theta)$ to $\sii(\Omega_0)$ defined by
\begin{equation}\label{unitary}
  U_\theta \Psi:=\Psi\circ\mathcal{L}_\theta
  \,.
\end{equation}
It is easy to check that
$
  H_{\theta} := U_\theta(-\Delta_D^{\Omega_\theta})U_\theta^{-1}
$
is the self-adjoint operator in $\sii(\Omega_0)$
associated with the quadratic form
\begin{equation}\label{form}
  Q_\theta[\psi]
  := \|\partial_1\psi-\dot{\theta}\,\partial_\tau\psi\|_{\sii(\Omega_0)}^2
  + \|\nabla'\psi\|_{\sii(\Omega_0)}^2
  \,, \qquad
  \psi \in \Dom(Q_\theta) := H_0^1(\Omega_0)
  \,.
\end{equation}
Here $\nabla':=(\partial_2,\partial_3)$ denotes
the transverse gradient and~$\partial_\tau$ is a shorthand for
the transverse angular-derivative operator
$$
  \partial_\tau := \tau\cdot\nabla'
  = x_3 \partial_2 - x_2 \partial_3
  \,,
  \qquad\mbox{where}\qquad
  \tau(x_2,x_3) := (x_3,-x_2)
  \,.
$$
We have the point-wise estimate
\begin{equation}\label{a-estimate}
  |\partial_\tau \psi| \leq a \, |\nabla' \psi|
  \,,\qquad\mbox{where}\qquad
  a := \sup_{x'\in\omega}|x'|
  \,.
\end{equation}
The sesquilinear form associated with $Q_\theta[\cdot]$
will be denoted by $Q_\theta(\cdot,\cdot)$.
In the distributional sense, we can write
\begin{equation}\label{H.distributional}
  H_{\theta} \psi
  = - (\partial_1-\dot{\theta}\,\partial_\tau)^2\psi - \Delta'\psi
  \,,
\end{equation}
where $-\Delta' := - \partial_2^2 - \partial_3^2$
denotes the transverse Laplacian.

\section{The Hardy and Sobolev inequalities}\label{Sec.Hardy}
%
In this section we summarize basic spectral results
about the Laplacian $-\Delta_D^{\Omega_\theta}$
we shall need later to study the asymptotic behaviour
of the associated semigroup.

\subsection{The Poincar\'e inequality}
%
Let~$E_1$ be the first eigenvalue
of the Dirichlet Laplacian in~$\omega$.
Using the Poincar\'e-type inequality in the cross-section
\begin{equation}\label{Poincare}
  \|\nabla f\|_{\sii(\omega)}^2 \geq E_1 \|f\|_{\sii(\omega)}^2
  \,, \qquad
  \forall f \in H_0^1(\omega)
  \,,
\end{equation}
and Fubini's theorem, it readily follows that
$
  Q_\theta[\psi] \geq E_1 \|\psi\|_{\sii(\Omega_0)}^2
$
for every $\psi \in H_0^1(\Omega_0)$.
Or, equivalently,
\begin{equation}\label{Poincare.tube}
  -\Delta_D^{\Omega_\theta} \geq E_1
\end{equation}
in the form sense in $\sii(\Omega_\theta)$.
Consequently, the spectrum of $-\Delta_D^{\Omega_\theta}$
does not start below~$E_1$.
The result~\eqref{Poincare.tube} can be interpreted
as a Poincar\'e-type inequality
and it holds for any tube~$\Omega_\theta$.

The inequality~\eqref{Poincare.tube} is clearly sharp
for an untwisted tube, since~\eqref{spectrum} holds
in that case trivially by separation of variables.
In general, the spectrum of $-\Delta_D^{\Omega_\theta}$
can start strictly above~$E_1$
if the twisting is effective at infinity
(\cf~\cite[Corol.~6.6]{K6-erratum}).
In this paper, however, we focus on tubes
for which the energy~$E_1$ coincides with
the spectral threshold of $-\Delta_D^{\Omega_\theta}$.
This is typically the case if the twisting vanishes at infinity
(\cf~\cite[Sec.~4]{K6}).
More restrictively, we assume~\eqref{locally}.
Under this hypothesis, \eqref{spectrum}~holds
and~\eqref{Poincare.tube} is sharp in the twisted case too.

\subsection{The Poincar\'e inequality in a bounded tube}
%
For our further purposes, it is important that
a better result than~\eqref{Poincare.tube}
holds in bounded tubes.

Given a bounded open interval $I \subset \Real$,
let~$H_\theta^I$ be the ``restriction'' of~$H_\theta$
to the tube $I \times \omega$ determined by the conditions
$
  \partial_1\psi-\dot\theta\partial_\tau\psi = 0
$
on the new boundary $(\partial I) \times \omega$.
More precisely, $H_\theta^I$~is introduced as the self-adjoint
operator in $\sii(I\times\omega)$ associated with the quadratic form
\begin{align*}
  Q_\theta^I[\psi]
  &:= \|\partial_1\psi-\dot{\theta}\,\partial_\tau\psi\|_{\sii(I\times\omega)}^2
  + \|\nabla'\psi\|_{\sii(I\times\omega)}^2
  \,,
  \\
  \psi \in \Dom(Q_\theta^I)
  &:=
  \left\{
  \psi\!\upharpoonright\!(I\times\Omega) \ | \ \psi \in H_0^1(\Omega_0)
  \right\}
  \,.
\end{align*}
That is, we impose no additional boundary conditions in the form setting.

Contrary to~$H_\theta$, $H_\theta^I$~is an operator with compact resolvent.
Let $\lambda(\dot\theta,I)$ denote the lowest eigenvalue
of the shifted operator $H_\theta^I-E_1$.
We have the following variational characterization:
\begin{equation}\label{lambda}
  \lambda(\dot\theta,I)
  = \min_{\psi \in \Dom(Q_{\theta}^I) \setminus\{0\}}
  \frac{\, Q_\theta^I[\psi]
  - E_1 \|\psi\|_{\sii(I\times\omega)}^2}
  {\|\psi\|_{\sii(I\times\omega)}^2}
  \,.
\end{equation}
As in the unbounded case, \eqref{Poincare}~yields that
$\lambda(\dot\theta,I)$ is non-negative
(it is zero if the tube is untwisted).
However, thanks to the compactness,
now we have that $H_\theta^I-E_1$ is a positive operator
whenever the tube is twisted.
\begin{Lemma}\label{Lem.cornerstone}
Let $\theta \in C^1(\Real)$.
Let $I \subset \Real$ be a bounded open interval
such that $\theta\!\upharpoonright\!I$ is not constant.
Let~$\omega$ be not rotationally invariant with respect
to the origin in~$\Real^2$.
Then
$$
  \lambda(\dot\theta,I) > 0
  \,.
$$
\end{Lemma}
\begin{proof}
We proceed by contradiction
and assume that $\lambda(\dot\theta,I) = 0$.
Then the minimum~\eqref{lambda} is attained
by a (smooth) function $\psi \in \Dom(Q_\theta^I)$
satisfying (recall~\eqref{Poincare})
\begin{equation}\label{2.ids}
  \|\partial_1 \psi-\dot\theta\,\partial_\tau \psi\|_{\sii(I\times\omega)}^2 = 0
  \quad\ \mbox{and}\ \quad
  \|\nabla'\psi\|_{\sii(I\times\omega)}^2
  - E_1 \|\psi\|_{\sii(I\times\omega)}^2 = 0
  \,.
\end{equation}
Writing
$
  \psi(x) = \varphi(x_1) \mathcal{J}_1(x') + \phi(x)
$,
where~$\mathcal{J}_1$ is the positive eigenfunction
corresponding to~$E_1$ of the Dirichlet Laplacian in $\sii(\omega)$
and $(\mathcal{J}_1,\phi(x_1,\cdot))_{\sii(\omega)}=0$
for every $x_1 \in I$,
we deduce from the second equality in~\eqref{2.ids}
that $\phi = 0$. The first identity is then
equivalent to
\begin{equation*}
   \|\dot\varphi\|_{\sii(I)}^2
   \|\mathcal{J}_1\|_{\sii(\omega)}^2
  + \|\dot\theta \;\! \varphi\|_{\sii(I)}^2
  \|\partial_\tau\mathcal{J}_1\|_{\sii(\omega)}^2
  \\
  - 2 (\mathcal{J}_1,\partial_\tau\mathcal{J}_1)_{\sii(\omega)}
  \Re (\dot\varphi,\dot\theta\;\!\varphi)_{\sii(I)}
  = 0
  \,.
\end{equation*}
Since $(\mathcal{J}_1,\partial_\tau\mathcal{J}_1)_{\sii(\omega)}=0$
by an integration by parts, it follows that~$\varphi$ must be constant
and that
$$
  \|\dot{\theta}\|_{\sii(I)} = 0
  \qquad\mbox{or}\qquad
  \|\partial_\tau\mathcal{J}_1\|_{\sii(\omega)} = 0
  \,.
$$
However, this is impossible under the stated assumptions because
$\|\dot\theta\|_{\sii(I)}$ vanishes if and only if
$\theta$~is constant on~$I$,
and $\partial_\tau\mathcal{J}_1 = 0$ identically in~$\omega$
if and only if $\omega$~is rotationally invariant with respect to the origin.
\end{proof}

Lemma~\ref{Lem.cornerstone} was the cornerstone
of the method of~\cite{K6-erratum} to establish
the existence of Hardy inequalities in twisted tubes
(see also the proof of Theorem~\ref{Thm.Hardy} below).

\subsection{Infinitesimally thin tubes}
%
It is clear that
$
  \lambda(\dot\theta,\Real)
  := \inf\sigma(H_\theta)
  =0
$
whenever~\eqref{spectrum} holds
(\eg, if~\eqref{locally} is satisfied).
It turns out that the shifted spectral threshold diminishes
also in the opposite asymptotic regime,
\ie\ when the interval $I_\epsilon:=(-\epsilon,\epsilon)$ shrinks,
and this irrespectively of the properties of~$\omega$ and $\dot\theta$.
\begin{Proposition}\label{Prop.erratum}
Let $\theta \in C^1(\Real)$.
We have
$$
  \lim_{\epsilon \to 0} \lambda\big(\dot\theta,I_\epsilon\big) = 0
  \,.
$$
\end{Proposition}
\begin{proof}
Let $\{\omega_k\}_{k=0}^\infty$ be an exhaustion sequence of~$\omega$,
\ie, each~$\omega_k$ is an open set with smooth boundaries
satisfying $\omega_k \Subset \omega_{k+1}$
and $\cup_{k=0}^\infty \omega_k = \omega$.
Let $\mathcal{J}_1^k$ denote the first eigenfunction
of the Dirichlet Laplacian in $\sii(\omega_k)$;
we extend it by zero to the whole~$\Real^2$.
Finally, set
$
  \psi^k := (1\otimes\mathcal{J}_1^k) \circ \mathcal{L}_{\theta_0}
$
with $\theta_0(x_1):=\dot\theta(0)\,x_1$, \ie,
$$
  \psi^k(x) = \mathcal{J}_1^k
  \left(
  x_2\cos(\dot\theta_0 x_1)+x_3\sin(\dot\theta_0 x_1),
  -x_2\sin(\dot\theta_0 x_1)+x_3\cos(\dot\theta_0 x_1)
  \right)
  \,,
$$
where $\dot\theta_0 := \dot\theta(0)$.

For any (large) $k \in \Nat$ there exists (small) positive~$\epsilon_k$
such that $\psi^k$ belongs to $\Dom(Q_\theta^{I_\epsilon})$
for all $\epsilon \leq \epsilon_k$.
Hence it is an admissible trial function for~\eqref{lambda}.
Now, fix $k \in \Nat$ and assume that $\epsilon \leq \epsilon_k$.
Then we have
\begin{equation*}
  \|\psi^k\|_{\sii(I_\epsilon\times\omega)}^2
  = |I_\epsilon| \, \|\mathcal{J}_1^k\|_{\sii(\omega_k)}^2
  \,,
\end{equation*}
where we have used the change of variables
$
  y = \mathcal{L}_{\theta_0}(x)
$.
At the same time, employing consecutively the identity
$
  \partial_1 \psi^k -\dot\theta\,\partial_\tau \psi^k
  = (\dot{\theta}_0-\dot\theta) \, \partial_\tau\psi^k
$,
the bound~\eqref{a-estimate}, the identity
$
  |\nabla'\psi^k(x)|
  = |\nabla \mathcal{J}_1^k(y_2,y_3)|
$
and the same change of variables,
we get the estimate
\begin{equation*}
  \|
  \partial_1 \psi^k -\dot\theta\,\partial_\tau \psi^k
  \|_{\sii(I_\epsilon\times\omega)}^2
  \leq
  \|
  (\dot\theta_0-\dot\theta)
  \|_{L^\infty(I_\epsilon)}^2
  \, |I_\epsilon| \, a^2 \,
  \|
  \nabla\mathcal{J}_1^k
  \|_{\sii(\omega_k)}^2
  \,,
\end{equation*}
where the supremum norm clearly tends to zero as $\epsilon \to 0$.
Finally,
$$
  \|\nabla'\psi^k\|_{\sii(I_\epsilon\times\omega)}^2
  - E_1 \|\psi^k\|_{\sii(I_\epsilon\times\omega)}^2
  = (E_1^k - E_1) \, |I_\epsilon| \, \|\mathcal{J}_1^k\|_{\sii(\omega_k)}^2
  \,,
$$
where~$E_1^k$ denotes the first eigenvalue
of the Dirichlet Laplacian in $\sii(\omega_k)$.
Sending~$\epsilon$ to zero, the trial-function argument therefore yields
$$
  \lim_{\epsilon \to 0}
  \lambda(\dot\theta,I_\epsilon) \leq E_1^k - E_1
  \,.
$$
Since~$k$ can be made arbitrarily large
and $E_1^k \to E_1$ as $k\to\infty$
by standard approximation arguments (see, \eg, \cite{Daners}),
we conclude with the desired result.
\end{proof}

\begin{Remark}[An erratum to \cite{K6}]
The study of the infinitesimally thin tubes
played a crucial in the proof of
Hardy inequalities given in~\cite{K6}.
According to Lemma~6.3 in~\cite{K6},
$\lambda\big(\dot\theta,I_\epsilon\big)$,
with constant~$\dot\theta$,
is independent of~$\epsilon>0$
(and therefore remains positive
for a twisted tube even if $\epsilon \to 0$).
However, in view Proposition~\ref{Prop.erratum}, this is false.
Consequently, Lemmata~6.3 and~6.5 and Theorem~6.6 in~\cite{K6} cannot hold.
The proof of Hardy inequalities presented in~\cite{K6} is incorrect.
A corrected version of the paper~\cite{K6} can be found in~\cite{K6-erratum}.
\end{Remark}
%

\subsection{The Hardy inequality}
%
Now we come back to unbounded tubes~$\Omega_\theta$.
Although~\eqref{Poincare.tube} represents a sharp
Poincar\'e-type inequality both for twisted and untwisted tubes
(if~\eqref{spectrum} holds),
there is a fine difference in the spectral setting.
Whenever the tube~$\Omega_\theta$ is non-trivially twisted
(\cf~Definition~\ref{definition}),
there exists a positive function~$\varrho$
(necessarily vanishing at infinity if~\eqref{spectrum} holds)
such that~\eqref{Poincare.tube} is improved to~\eqref{I.Hardy}.
A variant of the Hardy inequality
is represented by the following theorem:
\begin{Theorem}\label{Thm.Hardy}
Let $\theta \in C^1(\Real)$ and
suppose that $\dot\theta$ has compact support.
Then for every $\Psi \in H_0^1(\Omega_\theta)$ we have
\begin{equation}\label{Hardy}
  \|\nabla \Psi\|_{\sii(\Omega_\theta)}^2
  - E_1 \;\! \|\Psi\|_{\sii(\Omega_\theta)}^2
  \ \geq\
  c_H \, \|\rho \Psi\|_{\sii(\Omega_\theta)}^2
  \,,
\end{equation}
where $\rho(x):=1/\sqrt{1+x_1^2}$
and~$c_H$ is a non-negative constant
depending on~$\dot\theta$ and~$\omega$.
Moreover, $c_H$~is positive if, and only if,
$\Omega_\theta$ is twisted.
\end{Theorem}
\begin{proof}
It is clear that the left hand side of~\eqref{Hardy}
is non-negative due to~\eqref{Poincare.tube}.
The fact that $c_H=0$ if the tube is untwisted
follows from the more general result included in
Proposition~\ref{Prop.difference}.2 below.
We divide the proof of the converse fact
(\ie\ that twisting implies $c_H>0$)
into several steps.
Recall the identification of $\Psi\in\sii(\Omega_\theta)$
with $\psi:=U_\theta\Psi\in\sii(\Omega_0)$ via~\eqref{unitary}.

\smallskip
\noindent
\emph{1.}~Let us first assume that the interval
$I := (\inf\supp\dot\theta,\sup\supp\dot\theta)$
is symmetric with respect to the origin of~$\Real$.

\smallskip
\noindent
\emph{2.}~The main ingredient in the proof is the following
Hardy-type inequality for a Schr\"odinger operator
in $\Real\times\omega$
with a characteristic-function potential:
\begin{equation}\label{Hardy.classical}
  \|\rho\psi\|_{\sii(\Omega_0)}^2
  \leq 16 \, \|\partial_1\psi\|_{\sii(\Omega_0)}^2
  + (2+64/|I|^2) \, \|\psi\|_{\sii(I\times\omega)}^2
\end{equation}
for every $\psi \in H_0^1(\Omega_0)$.
This inequality is a consequence of the classical
one-dimensional Hardy inequality
$
  \int_\Real x_1^{-2} |\varphi(x_1)|^2 dx_1
  \leq 4 \int_\Real |\dot\varphi(x_1)|^2 dx_1
$
valid for any $\varphi\in H_0^1(\Real\!\setminus\!\{0\})$.
Indeed, following~\cite[Sec.~3.3]{EKK},
let~$\eta$ be the Lipschitz function on~$\Real$ defined by
$\eta(x_1):=2|x_1|/|I|$ for $|x_1|\leq |I|/2$ and~$1$ otherwise in~$\Real$
(we shall denote by the symbol the function $\eta \otimes 1$ on $\Real\times\omega$).
For any $\psi \in C_0^\infty(\Omega_0)$,
let us write $\psi =\eta\psi+(1-\eta)\psi$,
so that $(\eta\psi)(\cdot,x') \in H_0^1(\Real\!\setminus\!\{0\})$
for every $x'\in\omega$.
Then, employing Fubini's theorem, we can estimate as follows:
\begin{align*}
  \|\rho\psi\|_{\sii(\Omega_0)}^2
  & \leq 2 \int_{\Omega_0} x_1^{-2} \, |(\eta\psi)(x)|^2 \, dx
  + 2 \, \|(1-\eta)\psi\|_{\sii(\Omega_0)}^2
  \\
  & \leq 8 \, \|\partial_1(\eta\psi)\|_{\sii(\Omega_0)}^2
  + 2 \, \|\psi\|_{\sii(I\times\omega)}^2
  \\
  & \leq 16 \, \|\eta \partial_1\psi\|_{\sii(\Omega_0)}^2
  + 16 \, \|(\partial_1{\eta})\psi\|_{\sii(\Omega_0)}^2
  + 2 \, \|\psi\|_{\sii(I\times\omega)}^2
  \\
  & \leq 16 \, \|\partial_1\psi\|_{\sii(\Omega_0)}^2
  + (2+64/|I|^2) \, \|\psi\|_{\sii(I\times\omega)}^2
  \,.
\end{align*}
By density, this result extends to all
$\psi\in H_0^1(\Omega_0)=\Dom(Q_{\theta})$.

\smallskip
\noindent
\emph{3.}~By Lemma~\ref{Lem.cornerstone},
we have
\begin{equation}\label{bound1}
  Q_{\theta}[\psi] - E_1 \;\! \|\psi\|_{\sii(\Omega_0)}^2
  \, \geq \,
  Q_{\theta}^{I}[\psi] - E_1 \;\! \|\psi\|_{\sii(I\times\omega)}^2
  \, \geq \,
  \lambda(\dot\theta,I) \, \|\psi\|_{\sii(I\times\omega)}^2
\end{equation}
for every $\psi \in \Dom(Q_{\theta})$.
Here the first inequality employs the trivial fact that
the restriction to $I\times\omega$ of a function from
$\Dom(Q_{\theta})$ belongs to $\Dom(Q_{\theta}^I)$.
Under the stated hypotheses,
we know from Lemma~\ref{Lem.cornerstone} that
$\lambda(\dot\theta,I)$ is a positive number.

\smallskip
\noindent
\emph{4.}~At the same time, for every $\psi \in \Dom(Q_{\theta})$,
\begin{eqnarray}\label{bound3}
\lefteqn{
   Q_{\theta}[\psi] - E_1 \;\! \|\psi\|_{\sii(\Omega_0)}^2
   }
   \nonumber \\
   && \geq
   \epsilon \, \|\partial_1 \psi\|_{\sii(\Omega_0)}^2
   + \int_{\Omega_0} \left\{
   \left[1-\frac{\epsilon}{1-\epsilon} \, a^2 \, \dot{\theta}^2(x_1)\right]
   |\nabla'\psi(x)|^2 - E_1 \, |\psi(x)|^2 \right\} dx
   \nonumber \\
   && \geq
   \epsilon \, \|\partial_1 \psi\|_{\sii(\Omega_0)}^2
   - \frac{\epsilon}{1-\epsilon} \, a^2 E_1 \,
   \|\dot\theta \psi\|_{\sii(\Omega_0)}^2
   \nonumber \\
   && \geq
   \epsilon \, \|\partial_1 \psi\|_{\sii(\Omega_0)}^2
   - \frac{\epsilon}{1-\epsilon} \, \|\dot\theta\|_{L^\infty(I)}^2 \, a^2 E_1 \,
   \|\psi\|_{\sii(I\times\omega)}^2
\end{eqnarray}
for sufficiently small positive~$\epsilon$.
Here the first estimate is an elementary Cauchy-type inequality
employing~\eqref{a-estimate} and valid for all $\epsilon\in(0,1)$.
The second inequality in~\eqref{bound3} follows from~\eqref{Poincare}
with help of Fubini's theorem
provided that~$\epsilon$ is sufficiently small,
namely if $\epsilon < \big(1+a^2\|\dot{\theta}\|_{L^\infty(\Real)}^2\big)^{-1}$.

\smallskip
\noindent
\emph{5.}~Interpolating between the bounds~\eqref{bound1} and~\eqref{bound3},
and using~\eqref{Hardy.classical} in the latter,
we finally arrive at
\begin{multline*}
   Q_{\theta}[\psi] - E_1 \;\! \|\psi\|_{\sii(\Omega_0)}^2
   \geq
   \frac{1}{2}
   \frac{\epsilon}{16} \
   \|\rho\psi\|_{\sii(\Omega_0)}^2
   \\
   +
   \frac{1}{2}
   \left[
   \lambda(\dot\theta,I)
   - \epsilon \, \left(\frac{1}{8}+\frac{4}{|I|^2}\right)
   - \frac{\epsilon}{1-\epsilon} \, \|\dot\theta\|_{L^\infty(I)}^2 \, a^2 E_1
   \right]
   \|\psi\|_{\sii(I\times\omega)}^2
\end{multline*}
for every $\psi \in \Dom(Q_{\theta})$.
It is clear that the last line on the right hand side of this inequality
can be made non-negative by choosing~$\epsilon$ sufficiently small.
Such an~$\epsilon$ then determines the Hardy
constant $c_H' := \epsilon/32$.

\smallskip
\noindent
\emph{6.}~The previous bound can be transferred to $\sii(\Omega_\theta)$
via~\eqref{unitary}.
In general, if the centre of~$I$ is an arbitrary point $x_1^0\in\Real$,
the obtained result is equivalent to
$$
  \forall \Psi\in\Dom(Q_D^{\Omega_\theta})
  \,, \qquad
  \|\nabla \Psi\|_{\sii(\Omega_\theta)}^2
  - E_1 \;\! \|\Psi\|_{\sii(\Omega_\theta)}^2
  \ \geq\
  c_H' \, \|\rho_{x_1^0} \;\! \Psi\|_{\sii(\Omega_\theta)}^2
  \,,
$$
where $\rho_{x_1^0}(x):=1/\sqrt{1+(x_1-x_1^0)^2}$.
This yields~\eqref{Hardy} with
$$
  c_H := c_H' \min_{x_1 \in \Real} \frac{1+x_1^2}{1+(x_1-x_1^0)^2}
  \,,
$$
where the minimum is a positive constant depending on~$x_1^0$.
\end{proof}

The Hardy inequality of Theorem~\ref{Thm.Hardy}
was first established~\cite{EKK} under additional hypotheses.
The present version is adopted from~\cite{K6-erratum},
where other variants of the inequality can be found, too.

\subsection{The spectral stability}
%
Theorem~\ref{Thm.Hardy} provides certain stability properties
of the spectrum for twisted tubes,
while the untwisted case is always unstable,
in the following sense:
\begin{Proposition}\label{Prop.difference}
Let $V$ be the multiplication operator in $\sii(\Omega_\theta)$ by
a bounded non-zero non-negative function~$v$
such that $v(x) \sim |x_1|^{-2}$ as $|x_1|\to\infty$.
Then
\begin{enumerate}
\item
if~$\Omega_\theta$ is twisted with $\theta \in C^1(\Real)$
and $\dot\theta$ has compact support,
then there exists $\varepsilon_0>0$
such that for all $\varepsilon<\varepsilon_0$,
$$
  \inf\sigma(-\Delta_D^{\Omega_\theta}-\varepsilon V)
  \geq E_1 \,;
$$
\item
if~$\Omega_\theta$ is untwisted then, for all $\varepsilon>0$,
$$
  \inf\sigma(-\Delta_D^{\Omega_\theta}-\varepsilon V)
  < E_1 \,.
$$
\end{enumerate}
\end{Proposition}
\begin{proof}
The first statement follows readily from one part of Theorem~\ref{Thm.Hardy}.
To prove the second property
(and therefore the other part of Theorem~\ref{Thm.Hardy}
stating that $c_H=0$ if the tube is untwisted),
it is enough to consider the case $\theta=0$
and construct a test function~$\psi$ from $H_0^1(\Omega_0)$ such that
\begin{equation*}
  P_0[\psi] :=
  \|\nabla \psi\|_{\sii(\Omega_0)}^2 - E_1 \|\psi\|_{\sii(\Omega_0)}^2
  - \varepsilon \, \big\|v^{1/2}\psi\big\|_{\sii(\Omega_0)}^2
  < 0
\end{equation*}
for all positive~$\varepsilon$.
For every $n\geq 1$, we define
\begin{equation}\label{gef}
  \psi_n(x) := \varphi_n(x_1) \mathcal{J}_1(x_2,x_3)
  \,,
\end{equation}
where~$\mathcal{J}_1$ is the positive eigenfunction
corresponding to~$E_1$
of the Dirichlet Laplacian in the cross-section~$\omega$,
normalized to~$1$ in $\sii(\omega)$,
and
\begin{equation}\label{mollifier}
  \varphi_n(x_1) :=
  \exp{\left(-\frac{x_1^2}{n}\right)}
  \,.
\end{equation}
In view of the separation of variables
and the normalization of~$\mathcal{J}_1$,
we have
$$
  P_0[\psi_n] = \|\dot\varphi_n\|_{\sii(\Real)}^2
  - \varepsilon \, \big\|v_1^{1/2}\varphi_n\big\|_{\sii(\Real)}^2
  \,,
$$
where $v_1(x_1) := \| v(x_1,\cdot)^{1/2} \mathcal{J}_1\|_{\sii(\omega)}^2$.
By hypothesis, $v_1 \in L^1(\Real)$
and the integral $\|v_1\|_{L^1(\Real)}$ is positive.
Finally, an explicit calculation yields
$\|\dot\varphi_n\|_{\sii(\Real)} \sim n^{-1/4}$.
By the dominated convergence theorem, we therefore have
$$
  P_0[\psi_n] \xrightarrow[n\to\infty]{} - \varepsilon \, \|v_1\|_{L^1(\Real)}
  \,.
$$
Consequently, taking~$n$ sufficiently large and $\varepsilon$~positive,
we can make the form $P_0[\psi_n]$ negative.
\end{proof}

Since the potential~$V$ is bounded and vanishes at infinity,
it is easy to see that the essential spectrum is not changed,
\ie,
$
  \sigma_\mathrm{ess}(-\Delta_D^{\Omega_\theta}-\varepsilon V)
  = [E_1,\infty)
$,
independently of the value of~$\eps$
and irrespectively of whether the tube is twisted or not.
As a consequence of Proposition~\ref{Prop.difference},
we have that an arbitrarily small attractive potential $-\varepsilon V$
added to the shifted operator $-\Delta_D^{\Omega_\theta}-E_1$
in the untwisted tube would generate negative discrete eigenvalues,
however, a certain critical value of~$\varepsilon$ is needed
in order to generate the negative spectrum in the twisted case.
In the language of~\cite{Pinchover_2007},
the operator $-\Delta_D^{\Omega_\theta}-E_1$
is therefore subcritical (respectively critical)
if~$\Omega_\theta$ is twisted (respectively untwisted).

\subsection{An upper bound to the Hardy constant}
%
Now we come back to Theorem~\ref{Thm.Hardy} and show that
the Hardy weight at the right hand side of~\eqref{Hardy}
cannot be made arbitrarily large by increasing~$\dot\theta$
or making the cross-section~$\omega$ more eccentric.
\begin{Proposition}\label{Prop.limit}
Let $\theta \in C^1(\Real)$ and
suppose that $\dot\theta$ has compact support.
Then
$$
  c_H \leq 1/2 \,,
$$
where~$c_H$ is the constant of Theorem~\ref{Thm.Hardy}.
\end{Proposition}
\begin{proof}
Recall the unitary equivalence of~$-\Delta_D^{\Omega_\theta}$
and~$H_\theta$ given by~\eqref{unitary}.
We proceed by contradiction and show that the operator
$H_\theta - E_1 - c \rho^2$ is not non-negative if $c>1/2$,
irrespectively of properties of~$\theta$ and~$\omega$.
(Recall that~$\rho$ was initially introduced in Theorem~\ref{Thm.Hardy}
as a function on~$\Omega_\theta$.
In this proof, with an abuse of notation,
we denote by the same symbol analogous functions on~$\Omega_0$ and~$\Real$.)
It is enough to construct a test function~$\psi$
from~$\Dom(Q_\theta)$ such that
\begin{equation*}
  P_\theta^c[\psi] :=
  Q_\theta[\psi] - E_1 \|\psi\|_{\sii(\Omega_0)}^2
  - c \, \|\rho\psi\|_{\sii(\Omega_0)}^2
  < 0 \,.
\end{equation*}
As in the proof of Proposition~\ref{Prop.difference},
we use the decomposition~\eqref{gef},
but now the sequence of functions $\varphi_n$
is defined as follows:
$$
  \varphi_n(x_1) :=
  \begin{cases}
    \frac{x_1 - b_n^1}{b_n^2-b_n^1}
    & \mbox{if} \quad x_1 \in [b_n^1,b_n^2) \,,
    \\
    \frac{b_n^3 - x_1}{b_n^3-b_n^2}
    & \mbox{if} \quad x_1 \in [b_n^2,b_n^3) \,,
    \\
    0 & \mbox{otherwise} \,.
  \end{cases}
$$
Here $\{b_n^j\}_{n\in\Nat}$, with $j=1,2,3$,
are numerical sequences such that
$\sup\supp\dot\theta < b_n^1 < b_n^2 < b_n^3$
for each $n\in\Nat$ and $b_n^1 \to \infty$ as $n\to\infty$;
further requirements will be imposed later on.
Since~$\varphi_n$ and~$\dot\theta$ have disjoint supports,
and~$\mathcal{J}_1$ is supposed to be normalized to~$1$ in~$\sii(\omega)$,
it easily follows that
$$
  P_\theta^c[\psi_n]
  = \|\dot\varphi_n\|_{\sii(\Real)}^2
  - c \, \|\rho\varphi_n\|_{\sii(\Real)}^2
  \,.
$$
Note that the right hand side is independent of~$\theta$ and~$\omega$.
An explicit calculation yields
\begin{align*}
  \|\dot\varphi_n\|_{\sii(\Real)}^2
  =\ & \frac{1}{b_n^2-b_n^1} + \frac{1}{b_n^3-b_n^2} \,,
  \\
  \|\rho\varphi_n\|_{\sii(\Real)}^2
  =\ & \frac{
  b_n^2-b_n^1 + [(b_n^1)^2-1](\arctan b_n^2-\arctan b_n^1)
  -b_n^1 \log\frac{1+(b_n^2)^2}{1+(b_n^1)^2}}
  {(b_n^2-b_n^1)^2}
  \\
  & +
  \frac{
  b_n^3-b_n^2 + [(b_n^3)^2-1](\arctan b_n^3-\arctan b_n^2)
  -b_n^3 \log\frac{1+(b_n^3)^2}{1+(b_n^2)^2}}
  {(b_n^3-b_n^2)^2} \,.
\end{align*}
Specifying the numerical sequences in such a way that
also the quotients $b_n^2/b_n^1$ and $b_n^3/b_n^2$
tend to infinity as $n\to\infty$,
%
%
it is then straightforward to check that
$$
  b_n^2 \, P_\theta^c[\psi_n] \xrightarrow[n\to\infty]{} 1 - 2 c
  \,.
$$
Since the limit is negative for $c>1/2$, it follows
that $P_\theta^c[\psi_n]$ can be made negative
by choosing~$n$ sufficiently large.
\end{proof}
The proposition shows that the effect of twisting is limited
in its nature, at least if~\eqref{locally} holds.
This will have important consequences for the usage
of energy methods when studying the heat semigroup below.

\subsection{The Sobolev inequality}
%
Regardless of whether the tube is twisted or not,
the operator $-\Delta_D^{\Omega_\theta}-E_1$
satisfies the following Sobolev-type inequality.
\begin{Theorem}[Sobolev inequality]\label{Thm.Sobolev}
Let $\theta \in C^1(\Real)$ and
suppose that $\dot\theta$ has compact support.
Then for every $\Psi \in H_0^1(\Omega_\theta) \cap \sii(\Omega_\theta,\rho^{-2})$
we have
\begin{equation}\label{Sobolev}
  \|\nabla \Psi\|_{\sii(\Omega_\theta)}^2 - E_1 \|\Psi\|_{\sii(\Omega_\theta)}^2
  \ \geq\
  c_S \, \frac{\ \|\Psi\|_{\sii(\Omega_\theta)}^6}{\|\Psi\|_1^4}
  \,,
\end{equation}
where
$
  \|\Psi\|_1 := \sqrt{\int_\omega dx_2 dx_3
  \left(\int_\Real dx_1 |(\Psi\circ\mathcal{L}_\theta)(x)|\right)^2}
$
and~$c_S$ is a positive constant
depending on~$\dot\theta$ and~$\omega$.
\end{Theorem}
\begin{proof}
Recall that $\Psi\circ\mathcal{L}_\theta = U_\theta\Psi =:\psi$
belongs to $\sii(\Omega_0)$.
First of all, let us notice that $\|\Psi\|_1$ is well defined
for $\Psi \in \sii(\Omega_\theta,\rho^{-2})$.
Indeed, the Schwarz inequality together with Fubini's theorem yields
\begin{equation}\label{1norm}
  \|\Psi\|_1^2
  \leq \|\rho^{-1}\psi\|_{\sii(\Omega_0)}^2
  \int_\Real \frac{dx_1}{1+x_1^2}
  = \|\rho^{-1}\Psi\|_{\sii(\Omega_\theta)}^2 \, \pi
  < \infty
  \,.
\end{equation}
Here the equality of the norms is obvious from the facts that the mapping
$\mathcal{L}_\theta$ leaves invariant the first coordinate in~$\Real^3$
and that its Jacobian is one.
We also remark that, by density, it is enough to prove the theorem
for $\Psi \in C_0^\infty(\Omega_\theta)$.

The inequality~\eqref{Sobolev}
is a consequence of the one-dimensional inequality
\begin{equation}\label{Sobolev.1D}
  \forall \varphi \in H^1(\Real) \cap L^1(\Real)
  \,, \qquad
  \|\dot\varphi\|_{\sii(\Real)}^2
  \geq \frac{1}{4} \, \frac{\|\varphi\|_{\sii(\Real)}^6}{\|\varphi\|_{L^1(\Real)}^4}
  \,,
\end{equation}
which is established quite easily by combining
elementary estimates
$$
  \|\varphi\|_{\sii(\Real)}^2
  \leq \|\varphi\|_{L^1(\Real)} \|\varphi\|_{L^\infty(\Real)}
  \qquad\mbox{and}\qquad
  \|\varphi\|_{L^\infty(\Real)}^2
  \leq 2 \, \|\varphi\|_{L^2(\Real)} \|\dot\varphi\|_{L^2(\Real)}
  \,.
$$
In order to apply~\eqref{Sobolev.1D},
we need to estimate the left hand side of~\eqref{Sobolev}
from below by $\|\partial_1\psi\|_{\sii(\Omega_0)}^2$.
We can proceed as in the proof of Theorem~\ref{Thm.Hardy}.
Interpolating between the bounds~\eqref{bound1} and~\eqref{bound3},
we get
\begin{equation*}
  \|\nabla\Psi\|_{\sii(\Omega_\theta)}^2
  - E_1 \|\Psi\|_{\sii(\Omega_\theta)}^2
  \geq \frac{\epsilon}{2} \, \|\partial_1\psi\|_{\sii(\Omega_0)}^2
  \,,
\end{equation*}
where $\epsilon=:8 \, c_S$ is a positive constant
depending on~$\dot\theta$ and~$\omega$.
Using now~\eqref{Sobolev.1D} with help of Fubini's theorem,
we conclude the proof with
$$
  \|\partial_1\psi\|_{\sii(\Omega_0)}^2
  \geq \frac{1}{4} \int_\omega
  \frac{\|\psi(\cdot,x_2,x_3)\|_{\sii(\Real)}^6}
  {\|\psi(\cdot,x_2,x_3)\|_{L^1(\Real)}^4}
  \, dx_2 dx_3
  \geq \frac{1}{4} \frac{\ \|\Psi\|_{\sii(\Omega_\theta)}^6}{\|\Psi\|_1^4}
  \,.
$$
Here the second inequality follows by the H\"older inequality
with properly chosen conjugate exponents
(recall also that
$
  \|\psi\|_{\sii(\Omega_0)}=\|\Psi\|_{\sii(\Omega_\theta)}
$).
\end{proof}
%

\section{The energy estimates}\label{Sec.heat}
%
\subsection{The heat equation}
%
Having the replacement
$
  u(x,t) \mapsto e^{-E_1 t} \, u(x,t)
$
for~\eqref{I.heat} in mind,
let us consider the following $t$-time evolution problem
in the tube~$\Omega_\theta$:
\begin{equation}\label{heat}
  \left\{
\begin{aligned}
  u_t - \Delta u - E_1 u &= 0
  &\quad\mbox{in} & \quad \Omega_\theta\times(0,\infty) \,,
  \\
  u &= u_0
  &\quad\mbox{in} & \quad \Omega_\theta\times\{0\} \,,
  \\
  u &= 0
  &\quad\mbox{in} & \quad (\partial\Omega_\theta)\times(0,\infty) \,,
\end{aligned}
  \right.
\end{equation}
where $u_0 \in \sii(\Omega_\theta)$.

As usual, we consider the weak formulation of the problem,
\ie, we say a Hilbert space-valued function
$
  u \in \sii_\mathrm{loc}\big((0,\infty);H_0^1(\Omega_\theta)\big)
$,
with the weak derivative
$
  u' \in \sii_\mathrm{loc}\big((0,\infty);H^{-1}(\Omega_\theta)\big)
$,
is a (global) solution of~\eqref{heat} provided that
\begin{equation}\label{heat.weak}
  \big\langle v,u'(t)\big\rangle
  + \big(\nabla v,\nabla u(t)\big)_{\sii(\Omega_\theta)}
  - E_1 \, \big(v,u(t)\big)_{\sii(\Omega_\theta)} = 0
\end{equation}
for each $v \in H_0^1(\Omega_\theta)$ and a.e. $t\in[0,\infty)$,
and $u(0)=u_0$.
Here $\langle\cdot,\cdot\rangle$
denotes the pairing of $H_0^1(\Omega_\theta)$ and $H^{-1}(\Omega_\theta)$.
With an abuse of notation, we denote by the same symbol~$u$
both the function on $\Omega_\theta\times(0,\infty)$
and the mapping $(0,\infty) \to H_0^1(\Omega_\theta)$.

Standard semigroup theory implies that there indeed exists
a unique solution of~\eqref{heat} that belongs to
$C^0\big([0,\infty),\sii(\Omega_\theta)\big)$.
More precisely, the solution is given by $u(t) = S(t) u_0$,
where~$S(t)$ is the heat semigroup~\eqref{semigroup}
associated with $-\Delta_D^{\Omega_\theta}-E_1$.
By the Beurling-Deny criterion,
$S(t)$~is positivity-preserving for all~$t \geq 0$.

Since~$E_1$ corresponds to the threshold of
the spectrum of $-\Delta_D^{\Omega_\theta}$
if~\eqref{locally} holds,
we cannot expect a uniform decay of solutions of~\eqref{heat}
as $t\to\infty$ in this case.
More precisely, the spectral mapping theorem
together with~\eqref{spectrum} yields:
\begin{Proposition}\label{Prop.nodecay}
Let $\theta \in C^1(\Real)$ and
suppose that~$\dot\theta$ has compact support.
Then for each time $t \geq 0$ we have
$$
  \|S(t)\|_{\sii(\Omega_\theta)\to\sii(\Omega_\theta)} \, = \, 1
  \,.
$$
\end{Proposition}
\noindent
Consequently, for each $t>0$ and each $\varepsilon\in(0,1)$
we can find an initial datum $u_0 \in H_0^1(\Omega_\theta)$
such that $\|u_0\|_{\sii(\Omega_\theta)}=1$
and such that the solution of~\eqref{heat} satisfies
$$
  \|u(t)\|_{\sii(\Omega_\theta)} \ \geq \ 1-\varepsilon
  \,.
$$

\subsection{The dimensional decay rate}
%
However, if we restrict ourselves
to initial data decaying sufficiently fast
at the infinity of the tube, it is possible to obtain
a polynomial decay rate for the solutions of~\eqref{heat}.
In particular, we have the following result
based on Theorem~\ref{Thm.Sobolev}:
\begin{Theorem}\label{Thm.decay.1D}
Let $\theta \in C^1(\Real)$ and
suppose that~$\dot\theta$ has compact support.
Then for each time $t \geq 0$ we have
$$
  \|S(t)\|_{
  \sii(\Omega_\theta,\rho^{-2})
  \to
  \sii(\Omega_\theta)
  }
  \, \leq \,
  \left(
  1 + \frac{4 \, c_S}{\pi^2} \, t
  \right)^{-1/4}
  \,,
$$
where~$c_S$ is the positive constant of Theorem~\ref{Thm.Sobolev}
and~$\rho$ is introduced in Theorem~\ref{Thm.Hardy}.
\end{Theorem}
\begin{proof}
The statement is equivalent to the following bound
for the solution~$u$ of~\eqref{heat}:
\begin{equation}\label{dispersive}
  \forall t\in[0,\infty) \,, \qquad
  \|u(t)\|_{\sii(\Omega_\theta)}
  \ \leq \
  \|\rho^{-1}u_0\|_{\sii(\Omega_\theta)}
  \left(
  1 + \frac{4 \, c_S}{\pi^2} \, t
  \right)^{-1/4}
  \,,
\end{equation}
where
$
  u_0\in \sii(\Omega_\theta,\rho^{-2})
$
is any non-trivial datum.
It is easy to see that the real and imaginary parts of
the solution of~\eqref{heat} evolve separately.
Furthermore, since~$S(t)$ is positivity-preserving,
given a non-negative datum~$u_0$,
the solution~$u(t)$ remains non-negative for all $t \geq 0$.
Consequently, establishing the bound for
positive and negative parts of~$u(t)$ separately,
it is enough to prove~\eqref{dispersive} for
non-negative (and non-trivial) initial data only.
Without loss of generality,
we therefore assume in the proof below
that $u(t) \geq 0$ for all $t \geq 0$.

Let $\{\varphi_n\}_{n=1}^\infty$ be the family of mollifiers on~$\Real$
given by~\eqref{mollifier}; we denote by the same symbol
the functions $\varphi_n \otimes 1$
on $\Real\times\Real^2 \supset \Omega_\theta$.
Inserting the trial function
$$
  v_n(x;t) := \varphi_n(x_1) \, \bar{u}_n(x_2,x_3;t)
  \,, \qquad
  \bar{u}_n(x_2,x_3;t)
  := \big\|\varphi_n u(\cdot,x_2,x_3;t)\big\|_{\si(\Real)}
  \,,
$$
into~\eqref{heat.weak}, we arrive at
(recall the definition of $\|\cdot\|_1$ from Theorem~\ref{Thm.Sobolev})
\begin{align*}
  \frac{1}{2} \frac{d}{dt} \|\varphi_n u(t)\|_1^2
  &= - \|\nabla\bar{u}_n(t)\|_{\sii(\omega)}^2
  + E_1 \|\bar{u}_n(t)\|_{\sii(\omega)}^2
  - \big(\partial_1 v_n(t),\partial_1 u(t)\big)_{\sii(\Omega_\theta)}
  \\
  &\leq \big(\partial_1 v_n(t),\partial_1 u(t)\big)_{\sii(\Omega_\theta)}
  \\
  &\leq \|\partial_1 v_n(t)\|_{\sii(\Omega_\theta)}
  \|\nabla u(t)\|_{\sii(\Omega_\theta)}
  \,.
\end{align*}
Here the first inequality is due to the Poincar\'e-type inequality
in the cross-section~\eqref{Poincare}.
We clearly have
$$
  \|\partial_1 v_n(t)\|_{\sii(\Omega_\theta)}
  = \|\dot\varphi_n\|_{\sii(\Real)}
  \, \|\bar{u}_n(t)\|_{\sii(\omega)}
  = \|\dot\varphi_n\|_{\sii(\Real)}
  \, \|\varphi_n u(t)\|_1
  \,.
$$
Integrating the differential inequality,
we therefore get
$$
  \|\varphi_n u(t)\|_1 -\|\varphi_n u_0\|_1
  \leq \|\dot\varphi_n\|_{\sii(\Real)}
  \int_0^t \|\nabla u(t')\|_{\sii(\Omega_\theta)}^2 \, dt'
  \,.
$$
Since $\|\dot\varphi_n\|_{\sii(\Real)} \to 0$
and $\{\varphi_n\}_{n=1}^\infty$ is an increasing sequence of functions
converging pointwise to~$1$ as $n\to \infty$,
we conclude from this inequality that
\begin{equation}\label{mass}
  \forall t \in [0,\infty) \,, \qquad
  \|u(t)\|_1 \leq \|u_0\|_1
  \,,
\end{equation}
where $\|u_0\|_1$ is finite due to~\eqref{1norm}.

Now, substituting~$u$ for the trial function~$v$ in~\eqref{heat.weak},
applying Theorem~\ref{Thm.Sobolev} and using~\eqref{mass},
we get
\begin{align*}
  \frac{1}{2} \frac{d}{dt} \|u(t)\|_{\sii(\Omega_\theta)}^2
  &= - \Big(\|\nabla u(t)\|_{\sii(\Omega_\theta)}^2
  - E_1 \|u(t)\|_{\sii(\Omega_\theta)}^2\Big)
  \\
  &\leq - c_S \, \frac{\ \|u(t)\|_{\sii(\Omega_\theta)}^6}{\, \|u(t)\|_1^4}
  \\
  &\leq - c_S \, \frac{\ \|u(t)\|_{\sii(\Omega_\theta)}^6}{\, \|u_0\|_1^4}
  \,.
\end{align*}
An integration of this differential inequality leads to
\begin{align*}
  \forall t\in[0,\infty) \,, \qquad
  \|u(t)\|_{\sii(\Omega_\theta)}
  \ \leq \
  \|u_0\|_{\sii(\Omega_\theta)}
  \left(
  1+4 \, c_S \, \frac{\|u_0\|_{\sii(\Omega_\theta)}^4}{\|u_0\|_1^4} \ t
  \right)^{-1/4}
  .
\end{align*}

Dividing the last inequality by $\|\rho^{-1} u_0\|_{\sii(\Omega_\theta)}$
and replacing $\|u_0\|_1$ with $\|\rho^{-1} u_0\|_{\sii(\Omega_\theta)}$
using~\eqref{1norm}, we get
\begin{align*}
  \frac{\|u(t)\|_{\sii(\Omega_\theta)}}{\|\rho^{-1} u_0\|_{\sii(\Omega_\theta)}}
  \ \leq \
  \xi
  \left(
  1+ \frac{4 \, c_S}{\pi^2} \, \xi^4 \ t
  \right)^{-1/4}
  \leq
  \left(
  1+ \frac{4 \, c_S}{\pi^2} \ t
  \right)^{-1/4}
  ,
\end{align*}
where
$
  \xi := \|u_0\|_{\sii(\Omega_\theta)} / \|\rho^{-1} u_0\|_{\sii(\Omega_\theta)}
  \in (0,1)
$.
This establishes~\eqref{dispersive}.
\end{proof}

As a direct consequence of the theorem, we get:
\begin{Corollary}\label{Corol.decay.1D}
Under the hypotheses of Theorem~\ref{Thm.decay.1D},
$\Gamma(\Omega_\theta) \geq 1/4$.
\end{Corollary}
\begin{proof}
It is enough to realize that $\sii(\Omega_\theta,K)$
is embedded in $\sii(\Omega_\theta,\rho^{-2})$.
\end{proof}

The following proposition shows
that the decay rate of Theorem~\ref{Thm.decay.1D}
is optimal for untwisted tubes.
\begin{Proposition}\label{Prop.decay.1D}
Let $\Omega_\theta$ be untwisted.
Then for each time $t \geq 0$ we have
$$
  \|S(t)\|_{
  \sii(\Omega_\theta,K)
  \to
  \sii(\Omega_\theta)
  }
  \, \geq \,
  \frac{1}{\sqrt{2}} \,
  \left(
  1+t
  \right)^{-1/4}
  \,.
$$
\end{Proposition}
\begin{proof}
Without loss of generality, we may assume $\theta = 0$.
It is enough to find an initial datum
$u_0 \in \sii(\Omega_0,K)$
such that the solution~$u$ of~\eqref{heat} satisfies
\begin{equation}\label{norate}
  \forall t\in[0,\infty) \,, \qquad
  \frac{\|u(t)\|_{\sii(\Omega_0)}}{\|u_0\|_{\sii(\Omega_0,K)}}
  \ \geq \
  \frac{1}{\sqrt{2}} \,
  \left(
  1+t
  \right)^{-1/4}
  \,.
\end{equation}
The idea is to take $u_0 := \psi_n$,
where $\{\psi_n\}_{n=1}^\infty$ is the sequence~\eqref{gef}
approximating a generalized eigenfunction of $-\Delta_D^{\Omega_0}$
corresponding to the threshold energy~$E_1$.
Using the fact that~$\Omega_0$ is a cross-product
of~$\Real$ and~$\omega$,
\eqref{heat}~can be solved explicitly in terms
of an expansion into the eigenbasis
of the Dirichlet Laplacian in the cross-section
and a partial Fourier transform in the longitudinal variable.
In particular, for our initial data we get
$$
  \|u(t)\|_{\sii(\Omega_0)}^2
  = \int_\Real |\hat{\varphi}_n(\xi)|^2 \, \exp{(-2 \xi^2 t)} \, d\xi
  = \sqrt{\frac{n}{n+4t}}
  \sqrt{\frac{\pi n}{2}}
  \,,
$$
where the second equality is a result of an explicit calculation
enabled due to the special form of~$\varphi_n$ given by~\eqref{mollifier}.
At the same time, for every $n<8$
$\psi_n$~belongs to $\sii(\Omega_0,K)$
and an explicit calculation yields
$$
  \|u_0\|_{\sii(\Omega_0,K)}^2
  = 2 \, \sqrt{\frac{\pi n}{8-n}}
  \,.
$$
For the special choice $n=6$ we get that the left hand side
of~\eqref{norate} actually equals the right hand side
with~$t$ being replaced by $2t/3 < t$.
\end{proof}

The power~$1/4$ in the decay bounds of Theorem~\ref{Thm.decay.1D}
and Proposition~\ref{Prop.decay.1D} reflects the quasi-one-dimensional
nature of~$\Omega_\theta$ (\cf~\eqref{P.decay}),
at least if the tube is untwisted.
More precisely, Proposition~\ref{Prop.decay.1D} readily implies
that the inequality of Corollary~\ref{Corol.decay.1D} is sharp
for untwisted tubes.
\begin{Corollary}\label{Corol.norate}
Let $\Omega_\theta$ be untwisted.
Then $\Gamma(\Omega_\theta) = 1/4$.
\end{Corollary}

This result establishes one part of Theorem~\ref{Thm.rate}.
The much more difficult part is to show that the decay rate
is improved whenever the tube is twisted.

\subsection{The failure of the energy method}\label{Sec.failure}
%
As a consequence of combination of direct energy arguments
with Theorem~\ref{Thm.Hardy}, we get the following result.
In Remark~\ref{Rem.useless} below we explain why it is useless.
\begin{Theorem}\label{Thm.decay}
Let $\Omega_\theta$ be twisted with $\theta \in C^1(\Real)$.
Suppose that~$\dot\theta$ has compact support.
Then for each time $t \geq 0$ we have
\begin{equation}\label{decay}
  \|S(t)\|_{
  \sii(\Omega_\theta,\rho^{-2})
  \to
  \sii(\Omega_\theta)
  }
  \, \leq \,
  \left(
  1+2 \, t
  \right)^{\!-\min\{1/2,c_H/2\}} \
  \,,
\end{equation}
where~$c_H$ is the positive constant of Theorem~\ref{Thm.Hardy}.
\end{Theorem}
\begin{proof}
For any positive integer~$n$ and $x\in\Omega_\theta$,
let us set
$
  \rho_n(x) := \min\{\rho(x),n^{-1}\}
$.
Then $\{\rho_n^{-1}\}_{n=1}^\infty$
is a non-decreasing sequence of bounded functions
converging pointwise to~$\rho^{-1}$ as $n \to \infty$.
Recalling the definition of~$\rho$ from Theorem~\ref{Thm.Hardy},
it is clear that $x \mapsto \rho_n(x)$ is in fact
independent of the transverse variables~$x'$.
Moreover, $\rho_n^{-\gamma} u$ belongs to $H_0^1(\Omega_\theta)$
for every $\gamma\in\Real$ provided $u \in H_0^1(\Omega_\theta)$.

Choosing $v:=\rho_n^{-2} u$ in~\eqref{heat.weak}
(and possibly combining with the conjugate version
of the equation if we allow non-real initial data),
we get the identity
\begin{equation}\label{ineq1.5}
  \frac{1}{2} \frac{d}{dt} \|\rho_n^{-1}u(t)\|^2
  =
  -\|\rho_n^{-1}\nabla u(t)\|^2 + E_1 \|\rho_n^{-1}u(t)\|^2
  - \Re \Big(u(t)\nabla\rho_n^{-2},\nabla u(t)\Big)
  \,.
\end{equation}
Here and in the rest of the proof,
$\|\cdot\|$ and $(\cdot,\cdot)$ denote the norm and inner product
in $\sii(\Omega_\theta)$ (we suppress the subscripts in this proof).
Since~$\rho_n$ depends on the first variable only,
we clearly have
$
  \nabla(\rho_n^{-2})=(-2\rho^{-3}\partial_1\rho,0,0)
$.
Introducing an auxiliary function $v_n(t) := \rho_n^{-1} u(t)$,
one finds
\begin{equation*}
\begin{aligned}
  \|\rho_n^{-1}\nabla u(t)\|^2
  = \|\nabla v_n(t)\|^2
  + \|(\partial_1\rho_n/\rho_n) \, v_n(t)\|^2
  +2 \Re \Big(
  v_n(t), (\partial_1\rho_n/\rho_n) \, \partial_1 v_n(t)
  \Big)
  \,,
  \\
  \Re\Big(u(t)\nabla\rho_n^{-2},\nabla u(t)\Big)
  = -2 \|(\partial_1\rho_n/\rho_n) \, v_n(t)\|^2
  -2 \Re \Big(
  v_n(t),(\partial_1\rho_n/\rho_n) \, \partial_1 v_n(t)
  \Big)
  \,.
\end{aligned}
\end{equation*}
Combining these two identities
and substituting the explicit expression for~$\rho$,
we see that the right hand side of~\eqref{ineq1.5} equals
\begin{eqnarray}\label{explicit}
  \lefteqn{-\|\nabla v_n(t)\|^2 + E_1 \|v_n(t)\|^2
  + \|(\partial_1\rho_n/\rho_n) \, v_n(t)\|^2}
  \nonumber \\
  &&= -\|\nabla v_n(t)\|^2 + E_1 \|v_n(t)\|^2
  + \|\chi_n \rho v_n(t)\|^2 - \|\chi_n \rho^2 v_n(t)\|^2
  \\
  &&\leq (1-c_H) \, \|\chi_n \rho v_n(t)\|^2 - \|\chi_n \rho^2 v_n(t)\|^2
  \,.
  \nonumber
\end{eqnarray}
Here~$\chi_n$ denotes the characteristic function of the set
$\Omega_\theta^n:=\Omega_\theta \cap \{\supp(\partial_1\rho_n)\}$,
and the inequality follows from Theorem~\ref{Thm.Hardy}
and an obvious inclusion $\Omega_\theta^n \subset \Omega_\theta$.
Substituting back the solution~$u(t)$,
we finally arrive at
\begin{align}\label{ineq2}
  \frac{1}{2} \frac{d}{dt} \|\rho_n^{-1}u(t)\|^2
  &\leq
  (1-c_H) \, \|\chi_n \rho v_n(t)\|^2 - \|\chi_n \rho^2 v_n(t)\|^2
  \nonumber \\
  &\leq
  (1-c_H) \, \|\chi_n \rho v_n(t)\|^2
  \,.
\end{align}

Now, using the monotone convergence theorem
and recalling the initial data to which we restrict
in the hypotheses of the theorem,
the last estimate implies that $u(t)$ belongs to $\sii(\Omega_\theta,\rho^{-2})$
and that it remains true after passing to the limit $n\to\infty$,
\ie,
\begin{equation}\label{ineq2.bis}
  \frac{1}{2} \frac{d}{dt} \|\rho^{-1}u(t)\|^2
  \leq
  (1-c_H) \, \|u(t)\|^2
  \,.
\end{equation}
At the same time, we have
\begin{align}\label{ineq1}
  \frac{1}{2} \frac{d}{dt} \|u(t)\|^2
  &= - \Big(\|\nabla u(t)\|^2 - E_1 \|u(t)\|^2\Big)
  \nonumber \\
  &\leq - c_H \, \|\rho u(t)\|^2
  \nonumber \\
  &\leq - c_H \, \frac{\|u(t)\|^4}{\, \|\rho^{-1}u(t)\|^2}
  \,,
\end{align}
where the equality follows from~\eqref{heat},
the first inequality follows from Theorem~\ref{Thm.Hardy}
and the last inequality is established
by means of the Schwarz inequality.

Summing up, in view of~\eqref{ineq1} and~\eqref{ineq2.bis},
$a(t):=\|u(t)\|^2$ and $b(t):=\|\rho^{-1}u(t)\|^2$
satisfy the system of differential inequalities
\begin{equation}\label{system}
  \dot{a}
  \leq - 2 \, c_H \, \frac{a^2}{b}
  \,, \qquad
  \dot{b}
  \leq 2 \, (1-c_H) \, a
  \,,
\end{equation}
with the initial conditions
$a(0)=\|u_0\|^2=:a_0$ and $b(0)=\|\rho^{-1}u_0\|^2=:b_0$.
We distinguish two cases:
\smallskip \\
\emph{1.} \underline{$c_H \geq 1$}.
In this case, it follows from the second inequality of~\eqref{system}
that~$b$ is decreasing.
Solving the first inequality of~\eqref{system}
with~$b$ being replaced by~$b_0$, we then get
$$
  a(t) \leq a_0 \, \big[1+2 \, c_H \, (a_0/b_0) \, t \big]^{-1}
  \,.
$$
Dividing this inequality by~$b_0$
and maximizing the resulting right hand side
with respect to $a_0/b_0 \in (0,1)$,
we finally get
\begin{equation}
  \forall t\in[0,\infty) \,, \qquad
  \|u(t)\|
  \ \leq \
  \|\rho^{-1} u_0\|
  \left(
  1+2 \, c_H \, t
  \right)^{-1/2} \
  \,,
\end{equation}
which in particular implies~\eqref{decay}.
\smallskip \\
\emph{2.} \underline{$c_H \leq 1$}.
We ``linearize''~\eqref{system} by replacing one~$a$
of the square on the right hand side of first inequality
by employing the second inequality of~\eqref{system}:
$$
  \frac{\dot{a}}{a}
  \leq - 2 \, c_H \, \frac{a}{b}
  \leq -\frac{c_H}{1-c_H} \, \frac{\dot{b}}{b}
  \,.
$$
This leads to
$$
  a/a_0 \leq (b/b_0)^{-\frac{c_H}{1-c_H}} \,.
$$
Using this estimate in the original,
non-linearized system~\eqref{system},
\ie\ solving the system by eliminating~$b$
from the first and~$a$ from the second inequality of~\eqref{system},
we respectively obtain
$$
  a(t) \leq a_0 \, \big[1+2\,(a_0/b_0)\,t\big]^{-c_H}
  \,, \qquad
  b(t) \leq b_0 \, \big[1+2\,(a_0/b_0)\,t\big]^{1-c_H}
  \,.
$$
Dividing the first inequality by~$b_0$
and maximizing the resulting right hand side
with respect to $a_0/b_0 \in (0,1)$,
we finally get
\begin{equation}
  \forall t\in[0,\infty) \,, \qquad
  \|u(t)\|
  \ \leq \
  \|\rho^{-1} u_0\|
  \left(
  1+2 \, t
  \right)^{-c_H/2} \
  \,,
\end{equation}
which is equivalent to~\eqref{decay}.
\end{proof}
\begin{Remark}
We see that the power in the polynomial decay rate of Theorem~\ref{Thm.decay}
diminishes as $c_H \to 0$.
Let us now argue that this cannot be improved
by the present method of proof.
Indeed, the first inequality of~\eqref{ineq1}
is an application of the Hardy inequality of Theorem~\ref{Thm.Hardy}
and the second one is sharp.
The Hardy inequality is also applied
in the first inequality of~\eqref{ineq2}.
In the second inequality of~\eqref{ineq2}, however,
we have neglected a negative term.
Applying the second inequality of~\eqref{ineq1} to it instead,
we conclude with an improved system of differential inequalities
\begin{equation}\label{system.bis}
  \dot{a}
  \leq - 2 \, c_H \, \frac{a^2}{b}
  \,, \qquad
  \dot{b}
  \leq 2 \, (1-c_H) \, a - 2 \, \frac{a^2}{b}
  \,.
\end{equation}
The corresponding system of differential equations
has the explicit solution
\begin{equation*}
  \tilde{a}(t) = a_0 \left(\frac{\xi_0}
  {W\big[\xi_0 \exp{(\xi_0+2t)}\big]}\right)^{c_H}
  , \quad
  \tilde{b}(t) = a(t) \Big(
  1 + W\big[\xi_0 \exp{(\xi_0+2t)}\big]
  \Big)
  ,
\end{equation*}
where $\xi_0 := b_0/a_0-1>0$
and~$W$ denotes the Lambert W function (product log),
\ie~the inverse function of $w \mapsto w \exp(w)$.
Since
$$
  W\big[\xi_0 \exp{(\xi_0+2t)}\big] = 2\,t + o(t)
  \qquad\mbox{as}\qquad
  t \to \infty
  \,,
$$
we see that the $t^{-c_H/2}$ decay in~\eqref{decay} for $c_H<1$
cannot be improved
by replacing~\eqref{system} with~\eqref{system.bis}.
\end{Remark}
\begin{Remark}\label{Rem.useless}
Note that the hypothesis~\eqref{locally} is not explicitly
used in the proof of Theorem~\ref{Thm.decay},
it is just required that the inequality~\eqref{Hardy} holds
with some positive constant~$c_H$.
For tubes satisfying~\eqref{locally}, however,
we know from Proposition~\ref{Prop.limit}
that the constant cannot exceed the value $1/2$.
Consequently, irrespectively of the strength of twisting,
Theorem~\ref{Thm.decay} never represents an improvement
upon Theorem~\ref{Thm.decay.1D}.
This is what we mean by the failure of a direct energy argument
based on the Hardy inequality of Theorem~\ref{Thm.Hardy}.
\end{Remark}
%

\section{The self-similarity transformation}\label{Sec.similar}
%
Let us now turn to a completely different approach
which leads to an improved decay rate
regardless of the smallness of twisting.

\subsection{Straightening of the tube}\label{Sec.straightening}
%
First of all, we reconsider the heat equation~\eqref{heat}
in an untwisted tube~$\Omega_0$ by using the change of variables
defined by the mapping~$\mathcal{L}_\theta$.
In view of the unitary transform~\eqref{unitary},
one can identify the Dirichlet Laplacian in $\sii(\Omega_\theta)$
with the operator~\eqref{H.distributional} in $\sii(\Omega_0)$,
and it is readily seen that~\eqref{heat} is equivalent to
$$
  u_t + H_\theta u - E_1 u = 0
  \qquad \mbox{in} \qquad
  \Omega_0\times(0,\infty)
  \,,
$$
plus the Dirichlet boundary conditions on~$\partial\Omega_0$
and an initial condition at $t=0$.
(We keep the same latter~$u$ for the solutions
transformed to~$\Omega_0$.)
More precisely, the weak formulation~\eqref{heat.weak} is equivalent to
\begin{equation}\label{heat.weak.straight}
  \big\langle v,u'(t)\big\rangle
  + Q_\theta\big(v,u(t)\big)
  - E_1 \big( v,u(t)\big)_{\sii(\Omega_0)} = 0
\end{equation}
for each $v \in H_0^1(\Omega_0)$ and a.e.~$t\in[0,\infty)$,
with $u(0) = u_0 \in \sii(\Omega_0)$.
Here $\langle\cdot,\cdot\rangle$
denotes the pairing of $H_0^1(\Omega_0)$ and $H^{-1}(\Omega_0)$.
We know that the transformed solution~$u$
belongs to $C^0\big([0,\infty),\sii(\Omega_0)\big)$
by the semigroup theory.

\subsection{Changing the time}
%
The main idea is to adapt the method of self-similar solutions
used in the case of the heat equation in the whole Euclidean space
by Escobedo and Kavian~\cite{Escobedo-Kavian_1987}
to the present problem.
We perform the self-similarity transformation
in the first (longitudinal) space variable only,
while keeping the other (transverse) space variables unchanged.

More precisely, we consider a unitary transformation~$\tilde{U}$
on $\sii(\Omega_0)$ which associates to every solution
$
  u \in \sii_\mathrm{loc}\big((0,\infty),dt;\sii(\Omega_0,dx)\big)
$
of~\eqref{heat.weak.straight}
a self-similar solution~$\tilde{u}:=\tilde{U}u$
in a new $s$-time weighted space
$
  \sii_\mathrm{loc}\big((0,\infty),e^s ds;\sii(\Omega_0,dy)\big)
$
via~\eqref{SST}.
The inverse change of variables is given by
$$
  u(x_1,x_2,x_3,t)
  = (t+1)^{-1/4} \, \tilde{u}\big((t+1)^{-1/2}x_1,x_2,x_3,\log(t+1)\big)
  \,.
$$
When evolution is posed in that context,
$y=(y_1,y_2,y_3)$ plays the role of space variable and~$s$ is the new time.
One can check that, in the new variables,
the evolution is governed by~\eqref{heat.similar}.

More precisely, the weak formulation~\eqref{heat.weak.straight}
transfers to
\begin{equation}\label{heat.weak.similar}
  \big\langle
  \tilde{v}, \tilde{u}'(s)
  -\mbox{$\frac{1}{2}$} \, y_1 \;\! \partial_1\tilde{u}(s)
  \big\rangle
  + \tilde{Q}_{s}\big(\tilde{v},\tilde{u}(s)\big)
  - E_1 \, e^s \, \big(\tilde{v},\tilde{u}(s)\big)_{\sii(\Omega_0)} = 0
\end{equation}
for each $\tilde{v} \in H_0^1(\Omega_0)$ and a.e.~$s\in[0,\infty)$,
with $\tilde{u}(0) = \tilde{u}_0 := \tilde{U} u_0 = u_0$.
Here~$\tilde{Q}_{s}(\cdot,\cdot)$ denotes the sesquilinear form
associated with
\begin{align*}
  \tilde{Q}_{s}[\tilde{u}] &:=
  \|\partial_1\tilde{u}-\sigma_s\,\partial_\tau \tilde{u}\|_{\sii(\Omega_0)}^2
  + e^s \, \|\nabla'\tilde{u}\|_{\sii(\Omega_0)}^2
  - \frac{1}{4} \, \|\tilde{u}\|_{\sii(\Omega_0)}^2
  \,,
  \\
  \tilde{u} \in \Dom(\tilde{Q}_{s}) &:= H_0^1(\Omega_0)
  \,,
\end{align*}
where~$\sigma_s$ has been introduced in~\eqref{sigma}.

Note that the operator~$\tilde{H}_s$ in $\sii(\Omega_0)$
associated with the form~$\tilde{Q}_s$
has $s$-time-dependent coefficients,
which makes the problem different from the whole-space case.
In particular, the twisting represented by
the function~\eqref{sigma} becomes more and more
``localized'' in a neighbourhood of the origin $y_1=0$ for large time~$s$.

\subsection{The natural weighted space}
%
Since~$\tilde{U}$ acts as a unitary transformation on $\sii(\Omega_0)$,
it preserves the space norm of solutions
of~\eqref{heat.weak.straight} and~\eqref{heat.weak.similar}, \ie,
\begin{equation}\label{preserve}
  \|u(t)\|_{\sii(\Omega_0)}=\|\tilde{u}(s)\|_{\sii(\Omega_0)}
  \,.
\end{equation}
This means that we can analyse the asymptotic time behaviour
of the former by studying the latter.

However, the natural space to study the evolution~\eqref{heat.weak.similar}
is not $\sii(\Omega_0)$ but rather the weighted space~\eqref{weight}.
For $k \in \Int$, we define
$$
  \mathcal{H}_k :=
  \sii\big(\Omega_0,K^k(y_1) \, dy_1 dy_2 dy_3\big)
  \,.
$$
Hereafter we abuse the notation a bit by denoting by~$K$,
initially introduced as a function on~$\Omega_\theta$ in~\eqref{weight},
the analogous function on~$\Real$ too.
Note that~$K^{-1/2}$ is the first eigenfunction
of the harmonic-oscillator Hamiltonian
\begin{equation}\label{oscillator}
  h := -\frac{d^2}{dy_1^2} + \frac{1}{16} \, y_1^2
  \qquad \mbox{in} \qquad \sii(\Real)
\end{equation}
(\ie\ the Friedrichs extension
of this operator initially defined on $C_0^\infty(\Real)$).
The advantage of reformulating~\eqref{heat.weak.similar}
in~$\mathcal{H}_1$ instead of $\mathcal{H}_0 = \sii(\Omega_0)$
lies in the fact that then
the governing elliptic operator has compact resolvent,
as we shall see below (\cf~Proposition~\ref{Prop.compact}).

Let us also introduce the weighted Sobolev space
$$
  \mathcal{H}_k^1 :=
  H_0^1\big(\Omega_0,K^{k}(y_1) \, dy_1 dy_2 dy_3\big)
  \,,
$$
defined as the closure of $C_0^\infty(\Omega_0)$
with respect to the norm
$
  (\|\cdot\|_{\Hilbert_k}^2 + \|\nabla\cdot\|_{\Hilbert_k}^2)^{1/2}
$.
Finally, we denote by $\mathcal{H}_k^{-1}$
the dual space to $\Hilbert_k^1$.

\subsection{The evolution in the weighted space}
%
We want to reconsider~\eqref{heat.similar} as a parabolic problem
posed in the weighted space~$\Hilbert_1$ instead of~$\Hilbert_0$.
We begin with a formal calculation.
Choosing $\tilde{v}(y) := K(y_1) v(y)$
for the test function in~\eqref{heat.weak.similar},
where $v \in C_0^\infty(\Omega_0)$ is arbitrary,
we can formally cast~\eqref{heat.weak.similar}
into the form
\begin{equation}\label{heat.weak.weighted}
  \big\langle v, \tilde{u}'(s) \big\rangle
  + a_s\big(v,\tilde{u}(s)\big) = 0
  \,.
\end{equation}
Here $\langle\cdot,\cdot\rangle$
denotes the pairing of $\mathcal{H}_1^1$ and $\mathcal{H}_1^{-1}$,
and
\begin{align*}
  a_s(v,\tilde{u}) := & \
  \big(\partial_1 v - \sigma_s \, \partial_\tau v,
  \partial_1 \tilde{u} - \sigma_s \, \partial_\tau \tilde{u}
  \big)_{\mathcal{H}_1}
  + e^s \, \big(\nabla' v,\nabla'\tilde{u}\big)_{\Hilbert_1}
  \\
  & \
  - E_1 \, e^s \, \big(v,\tilde{u}\big)_{\mathcal{H}_1}
  - \frac{1}{2} \, \big(y_1 \;\! v,
  \sigma_s \, \partial_\tau \tilde{u}\big)_{\mathcal{H}_1}
  - \frac{1}{4} \, \big(v,\tilde{u}\big)_{\Hilbert_1}
  \,.
\end{align*}
Note that~$a_s$ is not a symmetric form.

Of course, the formulae are meaningless in general,
because the solution~$\tilde{u}(s)$ and its derivative~$\tilde{u}'(s)$
may not belong to $\mathcal{H}_1^1$ and $\mathcal{H}_1^{-1}$, respectively.
We therefore proceed conversely by showing that~\eqref{heat.weak.weighted}
is actually well posed in~$\mathcal{H}_1$
and that the solution solves~\eqref{heat.weak.similar} too.
As for the former, we have:
\begin{Proposition}\label{Prop.Lions}
For any $u_0 \in \Hilbert_1$,
there exists a unique function~$\tilde{u}$ such that
$$
  \tilde{u} \in \sii_\mathrm{loc}\big((0,\infty);\Hilbert_1^1\big)
  \cap C^0\big([0,\infty);\Hilbert_1\big)
  \,, \qquad
  \tilde{u}' \in \sii_\mathrm{loc}\big((0,\infty);\Hilbert_1^{-1}\big)
  \,,
$$
and it satisfies~\eqref{heat.weak.weighted}
for each $v \in \Hilbert_1^1$ and a.e.\ $s\in[0,\infty)$,
and $\tilde{u}(0)=u_0$.
\end{Proposition}
\begin{proof}
First of all, let us show that~$a_s$ is well-defined
as a sesquilinear form with domain $\Dom(a_s) := \Hilbert_1^1$
for any fixed $s \in [0,\infty)$.
In view of the boundedness of~$\sigma_s$ (for every finite~$s$)
and the estimate~\eqref{a-estimate},
it only requires to check that $y_1 v \in \Hilbert_1$
provided $v \in \Hilbert_1^1$.
Let $v \in C_0^\infty(\Omega_0)$. Then
\begin{align*}
  \|y_1 v\|_{\Hilbert_1}^2
  & = 2 \int_{\Omega_0} y_1 \, |v(y)|^2 \, \frac{d K(y_1)}{d y_1} \, dy
  \\
  & = -2 \int_{\Omega_0} \Big\{
  |v(y)|^2 + 2 \, y_1 \, \Re\big[\bar{v}(y)\partial_1 v(y)\big]
  \Big\} \, K(y_1) \, dy
  \\
  & \leq 4  \, \|y_1 \;\! v\|_{\Hilbert_1} \, \|\partial_1 v\|_{\Hilbert_1}
  \,.
\end{align*}
Consequently,
\begin{equation}\label{embed}
  \|y_1 v\|_{\Hilbert_1}
  \leq 4 \, \|\partial_1 v\|_{\Hilbert_1}
  \leq 4 \, \|v\|_{\Hilbert_1^1}
  \,.
\end{equation}
By density,
this inequality extends to all $v \in \Hilbert_1^1$.
Hence, $a_s(v,u)$ is well defined for all $s \geq 0$
and $v,u \in \Hilbert_1^1$ (we suppress the tilde over~$u$
in the rest of the proof).
Then the Proposition follows by a theorem of J.~L.~Lions
\cite[Thm.~X.9]{Brezis_FR}
about weak solutions of parabolic equations
with time-dependent coefficients.
We only need to verify its hypotheses:
\smallskip \\
\emph{1.~Measurability.}
The function $s \mapsto a_s(v,u)$ is clearly measurable
on $[0,\infty)$ for all $v,u \in \Hilbert_1^1$,
since it is in fact continuous.
\smallskip \\
\emph{2.~Boundedness.}
Let~$s_0$ be an arbitrary positive number.
Using the boundedness of~$\dot{\theta}$,
the estimates~\eqref{a-estimate} and~\eqref{embed},
it is quite easy to show that there is a constant~$C$,
depending uniquely on~$s_0$, $\|\dot\theta\|_{L^\infty(\Real)}$
and the geometry of~$\omega$ (through~$a$ and~$E_1$), such that
\begin{equation}\label{boundedness}
  |a_s(v,u)| \leq C \, \|v\|_{\Hilbert_1^1} \, \|u\|_{\Hilbert_1^1}
\end{equation}
for all $s \in [0,s_0]$ and $v,u \in \Hilbert_1^1$.
\smallskip \\
\emph{3.~Coercivity.}
Recall that~$a_s$ is not symmetric and that we consider complex
functional spaces. However, since the real and imaginary parts
of the solution~$\tilde{u}$ of~\eqref{heat.weak.weighted}
evolve independently, one may restrict to real-valued
functions~$v$ and~$\tilde{u}$ there.
Alternatively, it is enough to check the coercivity
of the real part of~$a_s$.
We therefore need to show that there are positive constants~$\epsilon$ and~$C$
such that the inequality
\begin{equation}\label{coercivity}
  \Re\{ a_s[v] \}
  \geq \epsilon \, \|v\|_{\Hilbert_1^1}^2
  - C \, \|v\|_{\Hilbert_1}^2
\end{equation}
holds for all $v \in \Hilbert_1^1$ and $s\in[0,s_0]$,
where $a_s[v] := a_s(v,v)$.
We have
\begin{multline}\label{begin.estimate}
  \Re\{ a_s[v] \} =
  \|\partial_1 v - \sigma_s \, \partial_\tau v \|_{\mathcal{H}_1}^2
  + e^s \, \|\nabla' v\|_{\mathcal{H}_1}^2
  - E_1 \, e^s \, \|v\|_{\mathcal{H}_1}^2
  - \frac{1}{4} \, \|v\|_{\mathcal{H}_1}^2
  \\
  - \frac{1}{2} \, \Re\,(y_1 \;\! v,
  \sigma_s \, \partial_\tau v)_{\mathcal{H}_1}
\end{multline}
for all $v \in \Hilbert_1^1$.
For every $v \in C_0^\infty(\Omega_0)$,
an integration by parts shows that:
\begin{equation}\label{mixed.term}
  \Re\,(y_1 \;\! v,\sigma_s \, \partial_\tau v)_{\mathcal{H}_1} = 0
  \,;
\end{equation}
by density, this result extends to all $v \in \Hilbert_1^1$.
Hence, the mixed term in~\eqref{begin.estimate} vanishes.
We continue with estimating the first term
on the right hand side of~\eqref{begin.estimate}:
\begin{align*}
  \|\partial_1 v - \sigma_s \, \partial_\tau v \|_{\mathcal{H}_1}^2
  &\geq \epsilon \, \|\partial_1 v\|_{\mathcal{H}_1}^2
  - \frac{\epsilon}{1-\epsilon} \, \|\sigma_s \, \partial_\tau v \|_{\mathcal{H}_1}^2
  \\
  &\geq \epsilon \, \|\partial_1 v\|_{\mathcal{H}_1}^2
  - \frac{\epsilon}{1-\epsilon} \, e^{s} \, \|\dot\theta\|_{L^\infty(\Real)}
  \, a^2 \, \|\nabla' v \|_{\mathcal{H}_1}^2
\end{align*}
valid for every $\epsilon\in(0,1)$ and $v \in \Hilbert_1^1$.
Here the second inequality follows from the definition
of~$\sigma_s$ in~\eqref{sigma} and the estimate~\eqref{a-estimate}.
Using~\eqref{Poincare} with help of Fubini's theorem,
we therefore have
\begin{multline*}
  \|\partial_1 v - \sigma_s \, \partial_\tau v \|_{\mathcal{H}_1}^2
  + (1-\epsilon) \, e^s \, \|\nabla' v\|_{\mathcal{H}_1}^2
  \\
  \geq \epsilon \, \|\partial_1 v\|_{\mathcal{H}_1}^2
  + E_1 \, e^s \left(
  1-\epsilon-\frac{\epsilon}{1-\epsilon}
  \, \|\dot\theta\|_{L^\infty(\Real)} \, a^2
  \right)
  \|v\|_{\mathcal{H}_1}^2
\end{multline*}
provided that~$\epsilon$ is sufficiently small
(so that the expression in the round brackets is positive).
Putting this inequality into~\eqref{begin.estimate},
recalling~\eqref{mixed.term}
and using the trivial bounds $1 \leq e^s \leq e^{s_0}$
for $s \in [0,s_0]$,
we conclude with
$$
  \Re\{ a_s[v] \} \geq \epsilon \, \|\nabla v\|_{\mathcal{H}_1}^2
  - \left[E_1 \, e^{s_0} \left(
  \epsilon+\frac{\epsilon}{1-\epsilon}
  \, \|\dot\theta\|_{L^\infty(\Real)} \, a^2
  \right)
  + \frac{1}{4}
  \right] \|v\|_{\mathcal{H}_1}^2
  \,,
$$
valid for all sufficiently small~$\epsilon$
and all real-valued $v \in \Hilbert_1^1$.
It is clear that the last inequality can be cast
into the form~\eqref{coercivity},
with a constant~$\epsilon$
depending on~$a$ and $\|\dot\theta\|_{L^\infty(\Real)}$,
and a constant~$C$ depending on
$s_0$, $\|\dot\theta\|_{L^\infty(\Real)}$
and the geometry of~$\omega$ (through $a$ and $E_1$).
\smallskip \\
Now it follows from~\cite[Thm.~X.9]{Brezis_FR}
that the unique solution~$\tilde{u}$ of~\eqref{heat.weak.weighted}
satisfies
$$
  \tilde{u} \in \sii\big((0,s_0);\Hilbert_1^1\big)
  \cap C^0\big([0,s_0];\Hilbert_1\big)
  \,, \qquad
  \tilde{u}' \in \sii\big((0,s_0);\Hilbert_1^{-1}\big)
  \,.
$$
Since~$s_0$ is an arbitrary positive number here,
we actually get a global continuous solution
in the sense that
$
  \tilde{u} \in
  C^0\big([0,\infty);\Hilbert_1\big)
$.
\end{proof}
\begin{Remark}
As a consequence of~\eqref{boundedness}, \eqref{coercivity}
and the Lax-Milgram theorem,
it follows that the form~$a_s$ is closed
on its domain $\Hilbert_1^1$.
\end{Remark}

Now we are in a position to prove a partial equivalence
of evolutions~\eqref{heat.weak.similar} and~\eqref{heat.weak.weighted}.
\begin{Proposition}
Let $u_0 \in \Hilbert_1$.
Let~$\tilde{u}$ be the unique solution to~\eqref{heat.weak.weighted}
for each $v \in \Hilbert_1^1$ and a.e.\ $s\in[0,\infty)$,
subject to the initial condition $\tilde{u}(0)=u_0$,
that is specified in Proposition~\ref{Prop.Lions}.
Then~$\tilde{u}$ is also the unique solution to~\eqref{heat.weak.similar}
for each $\tilde{v} \in \Hilbert_0^1$ and a.e.\ $s\in[0,\infty)$,
subject to the same initial condition.
\end{Proposition}
\begin{proof}
Choosing $v(y) := K(y_1)^{-1} \;\! \tilde{v}(y)$
for the test function in~\eqref{heat.weak.weighted},
where $\tilde{v} \in C_0^\infty(\Omega_0)$ is arbitrary,
one easily checks that~$\tilde{u}$ satisfies~\eqref{heat.weak.similar}
for each $\tilde{v} \in C_0^\infty(\Omega_0)$ and a.e.\ $s\in[0,\infty)$.
By density, this result extends to all $\tilde{v} \in \Hilbert_0^1$.
\end{proof}
%

\subsection{Reduction to a spectral problem}
%
As a consequence of the previous subsection,
reducing the space of initial data,
we can focus on the asymptotic time behaviour
of the solutions to~\eqref{heat.weak.weighted}.
Choosing $v := \tilde{u}(s)$ in~\eqref{heat.weak.weighted}
(and possibly combining with the conjugate version
of the equation if we allow non-real initial data),
we arrive at the identity
\begin{equation}\label{formal}
  \frac{1}{2} \frac{d}{ds} \|\tilde{u}(s)\|_{\mathcal{H}_{1}}^2
  = - J_s^{(1)}[\tilde{u}(s)]
  \,,
\end{equation}
where $J_s^{(1)}[\tilde{u}] := \Re\{a_s[\tilde{u}]\}$,
$\tilde{u} \in \Dom(J_s^{(1)}) := \Dom(a_s) = \Hilbert_1^1$
(independent of~$s$).
Recalling~\eqref{begin.estimate} and~\eqref{mixed.term},
we have
\begin{equation*}
  J_s^{(1)}[\tilde{u}] =
  \|\partial_1\tilde{u}-\sigma_s\,\partial_\tau \tilde{u}\|_{\Hilbert_1}^2
  + e^s \, \|\nabla'\tilde{u}\|_{\Hilbert_1}^2
  - E_1 \, e^s \, \|\tilde{u}\|_{\Hilbert_1}^2
  - \frac{1}{4} \, \|\tilde{u}\|_{\Hilbert_1}^2
  \,.
\end{equation*}
As a consequence of~\eqref{boundedness}, \eqref{coercivity}
and the Lax-Milgram theorem,
we know that $J_s^{(1)}$ is closed on its domain~$\Hilbert_1^1$.
It remains to analyse the coercivity of the form~$J_s^{(1)}$.

More precisely, as usual for energy estimates,
we replace the right hand side of~\eqref{formal}
by the spectral bound, valid for each fixed $s \in [0,\infty)$,
\begin{equation}\label{spectral.reduction}
  \forall \tilde{u} \in \Hilbert_1^1 \;\!, \qquad
  J_s^{(1)}[\tilde{u}]
  \geq \mu(s) \, \|\tilde{u}\|_{\Hilbert_1}^2
  \,,
\end{equation}
where~$\mu(s)$ denotes the lowest point in the spectrum of the self-adjoint
operator~$T_s^{(1)}$ in~$\Hilbert_1$ associated with~$J_s^{(1)}$.
Then~\eqref{formal} together with~\eqref{spectral.reduction} implies
the exponential bound
\begin{equation}\label{spectral.reduction.integral}
  \forall s \in [0,\infty) \;\!, \qquad
  \|\tilde{u}(s)\|_{\Hilbert_1}
  \leq \|\tilde{u}_0\|_{\Hilbert_1} \,
  e^{-\int_0^s \mu(r) dr}
  \,,
\end{equation}
In this way, the problem is reduced to a spectral analysis
of the family of operators $\{T_s^{(1)}\}_{s \geq 0}$.

\subsection{Removing the weight}
%
In order to investigate the operator~$T_s^{(1)}$ in~$\Hilbert_1$,
we first map it into a unitarily equivalent operator~$T_s^{(0)}$
in~$\Hilbert_0$.
This can be carried out via the unitary transform
$\mathcal{U}_{0}:\Hilbert_1\to\Hilbert_0$ defined by
$$
  (\mathcal{U}_{0}u)(y):=K^{1/2}(y_1)\,u(y)
  \,.
$$
We define $T_s^{(0)} := \mathcal{U}_0 T_s^{(1)} \mathcal{U}_0^{-1}$,
which is the self-adjoint operator associated with
the quadratic form $J_s^{(0)}[v] := J_s^{(1)}[\mathcal{U}_0^{-1}v]$,
$v \in \Dom(J_s^{(0)}) := \mathcal{U}_0\,\Dom(J_s^{(1)})$.
A straightforward calculation yields
\begin{equation}\label{J0.form}
  J_s^{(0)}[v]
  = \|\partial_1 v-\sigma_s\,\partial_\tau v\|_{\Hilbert_0}^2
  + \frac{1}{16} \, \|y_1 v\|_{\Hilbert_0}^2
  + e^s \, \|\nabla'v\|_{\Hilbert_0}^2
  - E_1 \;\! e^s \, \|v\|_{\Hilbert_0}^2
  \,.
\end{equation}

It is easy to verify that the domain of~$J_s^{(0)}$
coincides with the closure of $C_0^\infty(\Omega_0)$
with respect to the norm
$
  (\|\cdot\|_{\Hilbert_0}^2 + \|\nabla\cdot\|_{\Hilbert_0}^2
  +\|y_1\cdot\|_{\Hilbert_0}^2)^{1/2}
$.
In particular, $\Dom(J_s^{(0)})$~is independent of~$s$.
Moreover, since this closure is compactly embedded in~$\Hilbert_0$
(one can employ the well-known fact that~\eqref{oscillator}
has purely discrete spectrum, which essentially
uses the fact that the form domain of~$h$
is compactly embedded in $\sii(\Real)$),
it follows that $T_s^{(0)}$ (and therefore $T_s^{(1)}$)
is an operator with compact resolvent.
In particular, we have:
\begin{Proposition}\label{Prop.compact}
$T_s^{(1)} \simeq T_s^{(0)}$
have purely discrete spectrum for all $s\in[0,\infty)$.
\end{Proposition}
\noindent
Consequently, $\mu(s)$ is the lowest eigenvalue of $T_s^{(1)}$.

\subsection{The asymptotic behaviour of the spectrum}\label{Sec.strong}
%
In order to study the decay rate via~\eqref{spectral.reduction.integral},
we need information about the limit of the eigenvalue~$\mu(s)$
as the time~$s$ tends to infinity.

Since the function~$\sigma_s$ from~\eqref{sigma}
converges in the distributional sense to a multiple of the delta function
supported at zero as~$s\to\infty$,
it is expectable (\cf~\eqref{J0.form}) that the operator~$T_s^{(0)}$
will converge, in a suitable sense,
to the one-dimensional operator~$h$ from~\eqref{oscillator}
with an extra Dirichlet boundary condition at zero.
More precisely, the limiting operator, denoted by~$h_D$,
is introduced as the self-adjoint operator in~$\sii(\Real)$
whose quadratic form acts in the same way as that of~$h$
but has a smaller domain
$$
  \Dom(h_D^{1/2}) :=
  \big\{
  \varphi\in\Dom(h^{1/2})\ |\ \varphi(0)=0
  \big\}
  \,.
$$
Alternatively, the form domain $\Dom(h_D^{1/2})$ is the closure of
$C_0^\infty(\Real\setminus\{0\})$ with respect to the norm
$
  (\|\cdot\|_{\sii(\Real)}^2
  + \|\nabla\cdot\|_{\sii(\Real)}^2
  + \|y_1\cdot\|_{\sii(\Real)}^2)^{1/2}
$.

To make this limit rigorous
($T_s^{(0)}$ and~$h_D$ act in different spaces),
we follow~\cite{Friedlander-Solomyak_2007}
and decompose the Hilbert space~$\Hilbert_0$
into an orthogonal sum
$$
  \Hilbert_0 = \mathfrak{H}_1 \oplus \mathfrak{H}_1^\bot
  \,,
$$
where the subspace~$\mathfrak{H}_1$ consists of functions
of the form $\psi_1(y) = \varphi(y_1)\mathcal{J}_1(y')$.
Recall that~$\mathcal{J}_1$ denotes the positive
eigenfunction of $-\Delta_D^\omega$ corresponding to~$E_1$,
normalized to~$1$ in $\sii(\omega)$.
Given any $\psi \in \Hilbert_0$, we have the decomposition
$\psi = \psi_1 + \phi$ with $\psi_1\in\mathfrak{H}_1$ as above
and $\phi \in \mathfrak{H}_1^\bot$.
The mapping $\pi:\varphi\mapsto\psi_1$
is an isomorphism of $\sii(\Real)$ onto~$\mathfrak{H}_1$.
Hence, with an abuse of notations,
we may identify any operator~$h$ on $\sii(\Real)$
with the operator $\pi h \pi^{-1}$
acting on $\mathfrak{H}_1 \subset \Hilbert_0$.
\begin{Proposition}\label{Prop.strong}
Let $\Omega_\theta$ be twisted with $\theta \in C^1(\Real)$.
Suppose that~$\dot\theta$ has compact support.
Then $T_s^{(0)}$ converges to
$
  h_D \oplus 0^\bot
$
in the strong-resolvent sense as $s \to \infty$, \ie,
for every $F \in \Hilbert_0$,
$$
  \lim_{s \to \infty}
  \left\|
  \big(T_s^{(0)}+1\big)^{-1}F
  - \left[\big(h_D + 1 \big)^{-1} \oplus 0^\bot\right] F
  \right\|_{\Hilbert_0}
  = 0
  \,.
$$
Here~$0^\bot$ denotes the zero operator on the subspace
$\mathfrak{H}_1^\bot \subset \Hilbert_0$.
\end{Proposition}
\begin{proof}
For any fixed $F \in \Hilbert_0$
and sufficiently large positive number~$z$,
let us set $\psi_s := (T_s^{(0)}+z)^{-1}F$.
In other words, $\psi_s$~satisfies the resolvent equation
\begin{equation}\label{re}
  \forall v \in \Dom(J_s^{(0)}) \,, \qquad
  J_s^{(0)}(v,\psi_s) + z \, (v,\psi_s)_{\Hilbert_0}
  = (v,F)_{\Hilbert_0}
  \,.
\end{equation}
In particular, choosing~$\psi_s$ for the test function~$v$ in~\eqref{re},
we have
\begin{multline}\label{resolvent.identity}
  \|\partial_1 \psi_s-\sigma_s\,\partial_\tau \psi_s\|_{\Hilbert_0}^2
  + \frac{1}{16} \, \|y_1 \psi_s\|_{\Hilbert_0}^2
  + e^s \Big(
  \|\nabla'\psi_s\|_{\Hilbert_0}^2 - E_1 \|\psi_s\|_{\Hilbert_0}^2
  \Big)
  + z \, \|\psi_s\|_{\Hilbert_0}^2
  \\
  = (\psi_s,F)_{\Hilbert_0}
  \leq \frac{1}{4} \, \|\psi_s\|_{\Hilbert_0}^2 + \|F\|_{\Hilbert_0}^2
  \,.
\end{multline}
Henceforth we assume that $z > 1/4$.

We employ the decomposition
$\psi_s(y)=\varphi_s(y_1)\mathcal{J}_1(y_1) + \phi_s(y)$
where $\phi_s \in \mathfrak{H}_1^\bot$, \ie,
\begin{equation}\label{orthogonality}
  \forall y_1 \in \Real \,, \qquad
  \big(\mathcal{J}_1,\phi_s(y_1,\cdot)\big)_{\sii(\omega)} = 0
  \,.
\end{equation}
Then, for every $\epsilon \in (0,1)$,
\begin{align*}
  \|\nabla'\psi_s\|_{\Hilbert_0}^2 - E_1 \|\psi_s\|_{\Hilbert_0}^2
  &= \epsilon \|\nabla'\phi_s\|_{\Hilbert_0}^2
  + (1-\epsilon) \|\nabla'\phi_s\|_{\Hilbert_0}^2
  - E_1 \|\phi_s\|_{\Hilbert_0}^2
  \\
  &\geq \epsilon \|\nabla'\phi_s\|_{\Hilbert_0}^2
  + \big[(1-\epsilon)E_2-E_1\big] \|\phi_s\|_{\Hilbert_0}^2
  \,,
\end{align*}
where~$E_2$ denotes the second eigenvalue of~$-\Delta_D^\omega$.
Since~$E_1$ is (strictly) less then~$E_2$,
we can choose the~$\epsilon$ so small
that~\eqref{resolvent.identity} implies
\begin{equation}\label{ri1}
  \|\phi_s\|_{\Hilbert_0}^2 \leq C e^{-s}
  \qquad\mbox{and}\qquad
  \|\nabla'\phi_s\|_{\Hilbert_0}^2 \leq C e^{-s}
  \,,
\end{equation}
where~$C$ is a constant depending on~$\omega$ and $\|F\|_{\Hilbert_0}$.
At the same time, \eqref{resolvent.identity} yields
\begin{equation}\label{ri2}
  \|\varphi_s\|_{\sii(\Real)} \leq C \,, \qquad
  \|y_1\varphi_s\|_{\sii(\Real)} \leq C \,,
  \qquad\mbox{and}\qquad
  \|y_1\phi_s\|_{\Hilbert_0} \leq C
  \,,
\end{equation}
where~$C$ is a constant depending on~$\|F\|_{\Hilbert_0}$.

To get an estimate on the longitudinal derivative of~$\psi_s$,
we handle the first three terms on left hand side of~\eqref{resolvent.identity}
as follows.
Defining a new function $u_s\in\Hilbert_0$ by
$\psi_s(y)=e^{s/4} u_s(e^{s/2} y_1,y')$
(\cf~the self-similarity transformation~\eqref{SST})
and making the change of variables $(x_1,x')=(e^{s/2} y_1,y')$,
we have
\begin{align}\label{unself}
  J_s^{(0)}[\psi_s]
  &= e^s \|\partial_1 u_s-\dot\theta\,\partial_\tau u_s\|_{\Hilbert_0}^2
  + \frac{e^{-s}}{16} \, \|x_1 u_s\|_{\Hilbert_0}^2
  + e^s \Big(
  \|\nabla'u_s\|_{\Hilbert_0}^2 - E_1 \|u_s\|_{\Hilbert_0}^2
  \Big)
  \nonumber \\
  &\geq e^s \left\{
  \|\partial_1 u_s-\dot\theta\,\partial_\tau u_s\|_{\Hilbert_0}^2
  + \|\nabla'u_s\|_{\Hilbert_0}^2 - E_1 \|u_s\|_{\Hilbert_0}^2
  \right\}
  \nonumber \\
  &\geq e^s \, c_H \, \|\rho \;\! u_s\|_{\Hilbert_0}^2
  \,,
  \nonumber \\
  &= e^s \, c_H \, \|\rho_s\psi_s\|_{\Hilbert_0}^2
  \,, \qquad
  \mbox{where} \quad
  \rho_s(y) := \rho(e^{s/2}y_1,y') \,.
\end{align}
In the second inequality we have employed the Hardy inequality
of Theorem~\ref{Thm.Hardy}; the constant~$c_H$
is positive by the hypothesis.
Consequently, \eqref{resolvent.identity}~yields
\begin{equation}\label{ri3}
  \|\rho_s\psi_s\|_{\Hilbert_0}^2 \leq C e^{-s}
  \,,
\end{equation}
where~$C$ is a constant depending on~$\dot\theta$, $\omega$ and~$\|F\|_{\Hilbert_0}$.
Now, proceeding as in the proof of~\eqref{bound3}, we get
\begin{multline*}
  \|\partial_1 \psi_s-\sigma_s\,\partial_\tau \psi_s\|_{\Hilbert_0}^2
  + e^s \Big(
  \|\nabla'\psi_s\|_{\Hilbert_0}^2 - E_1 \|\psi_s\|_{\Hilbert_0}^2
  \Big)
  \\
  \geq  \epsilon \, \|\partial_1 \psi_s\|_{\Hilbert_{0}}^2
  - \frac{\epsilon}{1-\epsilon} \, \|\dot\theta\|_{L^\infty(\Real)}^2
  \, a^2 E_1 \, e^s \, \|\psi_s\|_{\sii(I_s\times\omega)}^2
\end{multline*}
for every $\epsilon < \big(1+a^2\|\dot{\theta}\|_{L^\infty(\Real)}^2\big)^{-1}$,
where
$
  I_s := e^{-s/2} I \equiv \{e^{-s/2} x_1 \,|\, x_1 \in I \}
$
with $I := (\inf\supp\dot\theta,\sup\supp\dot\theta)$.
Since
\begin{equation}\label{exclusively}
  \|\psi_s\|_{\sii(I_s\times\omega)}
  \leq
  C \, \|\rho_s\psi_s\|_{\Hilbert_0}
  \,,
\end{equation}
where~$C$ is a constant depending exclusively on~$I$,
\eqref{resolvent.identity}~together with~\eqref{ri3} implies
$
  \|\partial_1 \psi_s\|_{\Hilbert_{0}}^2
  \leq C
$,
where~$C$ is a constant depending on~$\dot\theta$, $\omega$ and~$\|F\|_{\Hilbert_0}$.
Recalling~\eqref{orthogonality}, we therefore get the separate bounds
\begin{equation}\label{ri4}
  \|\partial_1\phi_s\|_{\Hilbert_0} \leq C
  \qquad\mbox{and}\qquad
  \|\dot\varphi_s\|_{\sii(\Real)} \leq C
  \,,
\end{equation}
with the same constant~$C$.

By~\eqref{ri1}, $\phi_s$~converges strongly to zero
in~$\Hilbert_0$ as $s \to \infty$.
Moreover, it follows from~\eqref{ri1}, \eqref{ri2} and~\eqref{ri4}
that $\{\phi_s\}_{s \geq 0}$ is a bounded family in $\Dom(J_s^{(0)})$.
Consequently, $\phi_s$ converges weakly to zero
in $\Dom(J_s^{(0)})$ as $s \to \infty$.

At the same time, it follows from~\eqref{ri2} and~\eqref{ri4}
that $\{\varphi_s\}_{s \geq 0}$ is a bounded family in $\Dom(h^{1/2})$.
Therefore it is precompact in the weak topology of $\Dom(h^{1/2})$.
Let~$\varphi_\infty$ be a weak limit point,
\ie, for an increasing sequence of positive numbers $\{s_n\}_{n\in\Nat}$
such that $s_n \to \infty$ as $n \to \infty$,
$\{\varphi_{s_n}\}_{n\in\Nat}$
converges weakly to~$\varphi_\infty$ in $\Dom(h^{1/2})$.
Actually, we may assume that it converges strongly in $\sii(\Real)$
because $\Dom(h^{1/2})$ is compactly embedded in $\sii(\Real)$.

Employing~\eqref{orthogonality},
\eqref{ri3}~together with~\eqref{exclusively} gives
\begin{equation}\label{ri5}
  \|\varphi_s\|_{\sii(I_s)}^2 \leq C e^{-s}
  \,,
\end{equation}
where~$C$ is a constant depending on~$\dot\theta$, $\omega$ and~$\|F\|_{\Hilbert_0}$.
Multiplying this inequality by $e^{s/2}$
and taking the limit $s \to \infty$, we verify that
\begin{equation}\label{ri6}
  \varphi_\infty(0) = 0
  \,.
\end{equation}
(We note that $\Dom(h^{1/2}) \subset H^1(\Real)$
and that $H^1(J)$ is compactly embedded in $C^{0,\lambda}(J)$
for every $\lambda\in(0,1/2)$ and any bounded interval~$J\subset\Real$.)

Finally, let $\varphi \in C_0^\infty(\Real\!\setminus\!\{0\})$ be arbitrary.
Taking $v(x) := \varphi(x_1) \mathcal{J}_1(x')$
as the test function in~\eqref{re}, with~$s$ being replaced by~$s_n$,
and sending~$n$ to infinity, we easily check that
\begin{equation*}
  (\dot\varphi,\dot\varphi_\infty)_{\sii(\Real)}
  + \frac{1}{16} \, (y_1\varphi,y_1\varphi_\infty)_{\sii(\Real)}
  + z \, (\varphi,\varphi_\infty)_{\sii(\Real)}
  = (\varphi,f)_{\sii(\Real)}
  \,,
\end{equation*}
where $f(x_1) := (\mathcal{J}_1,F(x_1,\cdot))_{\sii(\omega)}$.
That is, $\varphi_\infty = (h_D+z)^{-1} f$,
for \emph{any} weak limit point of $\{\varphi_s\}_{s \geq 0}$.

Summing up, we have shown that~$\psi_{s}$
converges strongly to $\psi_\infty$
in $\Hilbert_0$ as $s \to \infty$,
where
$
  \psi_\infty(y)
  := \varphi_\infty(y_1)\mathcal{J}_1(y')
  = \big[(h_D+z)^{-1} \oplus 0^\bot \big] F
$.
\end{proof}
\begin{Remark}\label{Rem.unself}
The crucial step in the proof is certainly the usage
of the Hardy inequality in the second inequality of~\eqref{unself}.
Indeed, it enables one to control the mixed terms coming from
the first term on the left hand side of~\eqref{resolvent.identity}.
We would like to mention that instead of the Hardy inequality itself
we could have used in~\eqref{unself} the corner-stone Lemma~\ref{Lem.cornerstone}.
This would leave to the lower bound
$
  J_s^{(0)}[\psi_s] \geq
  e^s \, \lambda(\dot\theta,I) \, \|\psi_s\|_{\sii(I_s\times\omega)}^2
$,
which is sufficient to conclude the proof in the same way as above.
\end{Remark}
\begin{Corollary}\label{Corol.strong}
Let $\Omega_\theta$ be twisted with $\theta \in C^1(\Real)$.
Suppose that~$\dot\theta$ has compact support.
Then
$$
  \lim_{s\to\infty} \mu(s) = 3/4
  \,.
$$
\end{Corollary}
\begin{proof}
In general, the strong-resolvent convergence of Proposition~\ref{Prop.strong}
is not enough to guarantee the convergence of spectra.
However, in our case, since the spectra are purely discrete,
the eigenprojections converge even in norm (\cf~\cite{Weidmann_1980}).
In particular, $\mu(s)$~converges to the first eigenvalue of~$h_D$.
It remains to notice that the first eigenvalue of~$h_D$ coincides
(in view of the symmetry)
with the second eigenvalue of~$h$ which is~$3/4$.
(For the spectrum of~$h$, see any textbook dealing
with quantum harmonic oscillator, \eg, \cite[Sec.~2.3]{Griffiths}.)
\end{proof}
%

\subsection{The improved decay rate
- Proof of Theorem~\ref{Thm.rate}}\label{Sec.improved}
%
Now we have all the prerequisites to prove Theorem~\ref{Thm.rate}.
Recall that the identity $\Gamma(\Omega_\theta)=1/4$
for untwisted tubes is already established by Corollary~\ref{Corol.norate}.
Throughout this subsection we therefore assume
that~$\Omega_\theta$ is twisted with~\eqref{locally}
and show that there is an extra decay rate.

We come back to~\eqref{spectral.reduction.integral}.
It follows from Corollary~\ref{Corol.strong} that
for arbitrarily small positive number~$\eps$
there exists a (large) positive time~$s_\eps$ such that
for all $s \geq s_\eps$, we have $\mu(s) \geq 3/4 - \eps$.
Hence, fixing $\eps>0$, for all $s \geq s_\eps$, we have
$$
  {-\int_0^s \mu(r) \, dr}
  \leq {-\int_0^{s_\eps} \mu(r) \, dr} {-(3/4-\eps)(s-{s_\eps})}
  \leq {(3/4-\eps) s_\eps} {-(3/4-\eps) s}
  \,,
$$
where the second inequality is due to the fact
that~$\mu(s)$ is non-negative for all $s \geq 0$
(it is in fact greater than~$1/4$, \cf~Proposition~\ref{Prop.positivity}).
At the same time,  assuming $\eps \leq 3/4$, we trivially have
$$
  {-\int_0^s \mu(r) \, dr}
  \leq 0
  \leq {(3/4-\eps) s_\eps} {-(3/4-\eps) s}
$$
also for all $s \leq s_\eps$.
Summing up, \eqref{spectral.reduction.integral}~implies
\begin{equation}\label{instead}
  \|\tilde{u}(s)\|_{\mathcal{H}_{1}}
  \leq C_\eps \, e^{-(3/4-\eps)s} \, \|\tilde{u}_0\|_{\mathcal{H}_{1}}
\end{equation}
for every $s \in [0,\infty)$,
where $C_\eps := e^{s_\eps} \geq e^{(3/4-\eps)s_\eps}$.
Returning to the variables in the straightened tube
via $u = \tilde{U}^{-1} \tilde{u}$,
using~\eqref{preserve} together with the point-wise estimate $1 \leq K$,
and recalling that $\tilde{u}_0=u_0$,
it follows that
$$
  \|u(t)\|_{\Hilbert_0}
  = \|\tilde{u}(s)\|_{\Hilbert_0}
  \leq \|\tilde{u}(s)\|_{\Hilbert_{1}}
  \leq C_\eps \, (1+t)^{-(3/4-\eps)} \, \|u_0\|_{\mathcal{H}_{1}}
$$
for every $t \in [0,\infty)$.
Finally, we recall that the weight~$K$ in~$\Hilbert_1$
depends on the longitudinal variable only,
which is therefore left invariant by the mapping~$\mathcal{L}_\theta$.
Consequently, we apply the unitary transform~\eqref{unitary}
and conclude with
$$
  \|S(t)\|_{\sii(\Omega_\theta,K) \to \sii(\Omega_\theta)}
  = \sup_{u_0 \in \Hilbert_1\setminus\{0\}}
  \frac{\|u(t)\|_{\Hilbert_0}}{\|u_0\|_{\Hilbert_1}}
  \leq C_\eps \, (1+t)^{-(3/4-\eps)}
$$
for every $t \in [0,\infty)$.
Since~$\eps$ can be made arbitrarily small,
this bound implies $\Gamma(\Omega_\theta) \geq 3/4$
and concludes thus the proof of Theorem~\ref{Thm.rate}.

\subsection{The improved decay rate
- an alternative statement}\label{Sec.alternative}
%
Theorem~\ref{Thm.rate} provides quite precise information
about the extra polynomial decay of solutions~$u$ of~\eqref{I.heat}
in a twisted tube in the sense that the decay rate~$\Gamma(\Omega_\theta)$
is at least three times better than in the untwisted case.
On the other hand, we have no control over
the constant~$C_\Gamma$ in~\eqref{solution.rate}
(in principle it may blow up as $\Gamma \to \Gamma(\Omega_\theta)$).
As an alternative result, we therefore present also the following theorem,
where we get rid of the constant~$C_\Gamma$
but the prize we pay is just a qualitative knowledge
about the decay rate.
\begin{Theorem}\label{Thm.I}
Let $\theta\in C^1(\Real)$ satisfy~\eqref{locally}.
We have
\begin{equation}\label{decay.similar.I}
  \forall t \geq 0 \,, \qquad
  \|S(t)\|_{
  \sii(\Omega_\theta,K)
  \to
  \sii(\Omega_\theta)
  }
  \, \leq \,
  \left(
  1+t
  \right)^{\!-(\gamma+1/4)}
  \,,
\end{equation}
where~$\gamma$ is a non-negative constant
depending on~$\dot{\theta}$ and~$\omega$.
Moreover, $\gamma$~is positive if, and only if,
$\Omega_\theta$ is twisted.
\end{Theorem}

In order to establish Theorem~\ref{Thm.I},
the asymptotic result of Corollary~\ref{Corol.strong}
need to be supplied with information
about values of~$\mu(s)$ for finite times~$s$.

\subsubsection{Singling the dimensional decay rate out}
%
It follows from Theorem~\ref{Thm.decay.1D} that
there is at least a $1/4$ polynomial decay rate
for the solutions of the heat equations.
In the setting of self-similar solutions
(recall~\eqref{spectral.reduction.integral}
and the relation between the initial and self-similar
times~$t$ and~$s$ given by~\eqref{SST}),
this will be reflected in that we actually have $\mu(s) \geq 1/4$,
regardless whether the tube is twisted or not.
It is therefore natural to study rather
the shifted operator $T_s^{(0)}-1/4$.
However, it is not obvious from~\eqref{J0.form}
that such an operator is non-negative.

In order to introduce the shift explicitly
into the structure of the operator,
we therefore introduce another unitarily equivalent operator
$T_s^{(-1)}:=\mathcal{U}_{-1} T_s^{(0)} (\mathcal{U}_{-1})^{-1}$ in~$\Hilbert_{-1}$,
where the map $\mathcal{U}_{-1} : \Hilbert_0 \to \Hilbert_{-1}$
acts in the same way as~$\mathcal{U}_0$:
$$
  (\mathcal{U}_{-1}v)(y):=K^{1/2}(y_1)\,v(y)
  \,.
$$
$T_s^{(-1)}$ is the self-adjoint operator associated with
the quadratic form $J_s^{(-1)}[w] := J_s^{(0)}[(\mathcal{U}_{-1})^{-1}w]$,
$w \in \Dom(J_s^{(-1)}) := \mathcal{U}_{-1}\,\Dom(J_s^{(0)})$.
Again, it is straightforward to check that
\begin{align*}
  J_s^{(-1)}[w]
  &= \|\partial_1 w-\sigma_s\,\partial_\tau w\|_{\Hilbert_{-1}}^2
  + e^s \, \|\nabla'w\|_{\Hilbert_{-1}}^2
  - E_1 \, e^s \, \|w\|_{\Hilbert_{-1}}^2
  + \frac{1}{4} \, \|w\|_{\Hilbert_{-1}}^2
  \,.
\end{align*}

Now it readily follows from the structure of the quadratic form
that the shifted operator $T_s^{(-1)}-1/4$ is non-negative.
Moreover, it is positive if, and only if, the tube is twisted.
\begin{Proposition}\label{Prop.positivity}
If $\Omega_\theta$ is twisted with $\theta \in C^1(\Real)$,
then we have
$$
  \forall s \in [0,\infty), \qquad
  \mu(s) > 1/4
  \,.
$$
Conversely, $\mu(s) = 1/4$ for all $s \in [0,\infty)$
if $\Omega_\theta$ is untwisted.
\end{Proposition}
\begin{proof}
Since $J_s^{(-1)}[w] - \frac{1}{4} \, \|w\|_{\Hilbert_{-1}}^2  \geq 0$
for every $w \in \Dom(J_s^{(-1)})$, we clearly have $\mu(s) \geq 1/4$,
regardless whether the tube is twisted or not.
By definition, if it is untwisted,
then either $\sigma_s = 0$ identically in~$\Real$
for all $s \in [0,\infty)$
or $\partial_\tau \mathcal{J}_1 = 0$ identically in~$\omega$,
where~$\mathcal{J}_1$ is the positive eigenfunction corresponding to~$E_1$
of the Dirichlet Laplacian in~$\sii(\omega)$.
Consequently, choosing $w(y) = \mathcal{J}_1(y')$
as a test function for~$J_s^{(-1)}$,
we also get the opposite bound $\mu(s) \leq 1/4$
in the untwisted case.
To get the converse result, we can proceed exactly
as in the proof of Lemma~\ref{Lem.cornerstone}:
Assuming $\mu(s) = 1/4$ in the twisted case,
the variational definition of the eigenvalue~$\mu(s)$ would imply
$$
  \|\sigma_s\|_{\sii(\Real,K^{-1})} = 0
  \qquad\mbox{or}\qquad
  \|\partial_\tau\mathcal{J}_1\|_{\sii(\omega)} = 0
  \,,
$$
a contradiction.
\end{proof}

Now we are in a position to prove Theorem~\ref{Thm.I}.

\subsubsection{Proof of Theorem~\ref{Thm.I}}
%
Assume~\eqref{locally}.
It follows from Proposition~\ref{Prop.positivity}
and Corollary~\ref{Corol.strong} that the number
\begin{equation}
  \gamma := \inf_{s \in [0,\infty)}\mu(s) - 1/4
\end{equation}
is positive if, and only if, $\Omega_\theta$~is twisted.
In any case, \eqref{spectral.reduction.integral}~implies
$$
  \|\tilde{u}(s)\|_{\mathcal{H}_{1}}
  \leq \|\tilde{u}_0\|_{\mathcal{H}_{1}} \, e^{-(\gamma+1/4)s}
$$
for every $s \in [0,\infty)$.
Using this estimate instead of~\eqref{instead},
but following the same type of arguments as in Section~\ref{Sec.improved}
below~\eqref{instead}, we get
$$
  \|S(t)\|_{\sii(\Omega_\theta,K) \to \sii(\Omega_\theta)}
  \leq (1+t)^{-(\gamma+1/4)}
$$
for every $t \in [0,\infty)$.
This is equivalent to~\eqref{decay.similar.I}
and we know that~$\gamma$ is positive if~$\Omega_\theta$ is twisted.
On the other hand, in view of Proposition~\ref{Prop.decay.1D},
estimate~\eqref{decay.similar.I} cannot hold with positive~$\gamma$
if the tube is untwisted.
This concludes the proof of Theorem~\ref{Thm.I}.

\section{Conclusions}\label{Sec.end}
%
The classical interpretation of the heat equation~\eqref{I.heat}
is that its solution~$u$ gives the evolution of the temperature distribution
of a medium in the tube cooled down to zero on the boundary.
It also represents the simplest version of the stochastic Fokker-Planck equation
describing the Brownian motion in~$\Omega_\theta$
with killing boundary conditions.
Then the results of the present paper can be interpreted
as that the twisting implies a faster cool-down/death
of the medium/Brownian particle in the tube.
Many other diffusive processes in nature are governed by~\eqref{I.heat}.

Our proof that there is an extra decay rate for solutions of~\eqref{I.heat}
if the tube is twisted was far from being straightforward.
This is a bit surprising because the result is quite
expectable from the physical interpretation,
if one notices that the twist (locally) enlarges
the boundary of the tube, while it (locally) keeps the volume unchanged.
(By ``locally'' we mean that it is the case for bounded tubes,
otherwise both the quantities are infinite of course.)
At the same time, the Hardy inequality~\eqref{I.Hardy}
did not play a direct role in the proof of
Theorems~\ref{Thm.rate} and~\ref{Thm.I}
(although, combining any of the theorems with Theorem~\ref{Thm.Hardy},
we eventually know that the existence of the Hardy inequality
is equivalent to the extra decay rate for the heat semigroup).
It would be desirable to find a more direct proof
of Theorem~\ref{Thm.rate} based on~\eqref{I.Hardy}.

We conjecture that the inequality of Theorem~\ref{Thm.rate}
can be replaced by equality, \ie, $\Gamma(\Omega_\theta)=3/4$
if the tube is twisted and~\eqref{locally} holds.
The study of the quantitative dependence of the constant~$\gamma$
from Theorem~\ref{Thm.I} on properties of~$\dot\theta$
and the geometry of~$\omega$ also constitutes an interesting open problem.
Note that the two quantities are related by
$\gamma+1/4 \leq \Gamma(\Omega_\theta)$.

Throughout the paper we assumed~\eqref{locally}.
We expect that this hypothesis can be replaced
by a mere vanishing of~$\dot\theta$ at infinity
to get Theorems~\ref{Thm.rate} and~\ref{Thm.I}
(and also Theorem~\ref{Thm.Hardy}).
This less restrictive assumption is known to be enough to ensure~\eqref{spectrum}
and there exist versions of~\eqref{I.Hardy}
even if~\eqref{locally} is violated (\cf~\cite{K6-erratum}).
However, it is quite possible that a slower decay of~$\dot\theta$
at infinity will make the effect of twisting stronger.
In particular, can~$\Gamma(\Omega_\theta)$ be strictly greater than~$3/4$
if the tube is twisted and~$\dot\theta$
decays to zero very slowly at infinity?

Equally, it is not clear whether Proposition~\ref{Prop.limit} holds
if~\eqref{locally} is violated.
There are some further open problems related to the Hardy inequality
of Theorem~\ref{Thm.Hardy}.
In particular, it is frustrating that the proof of the theorem
does not extend to all~$\dot\theta$ merely vanishing at infinity.
In this context, it would be highly desirable to establish
a more quantitative version of Lemma~\ref{Lem.cornerstone},
\ie~to get a positive lower bound to $\lambda(\dot{\theta},I)$
depending explicitly on~$\dot\theta$, $|I|$ and~$\omega$.

On the other hand, a completely different situation will appear
if one allows twisted tubes for which~$\dot\theta$
does not vanish at infinity.
Then the spectrum of $-\Delta_D^{\Omega_\theta}$
can actually start strictly above~$E_1$
(\cf~\cite{EKov_2005} or \cite[Corol.~6.6]{K6})
and an extra exponential decay rate for our semigroup~$S(t)$
follows at once already in $\sii(\Omega_\theta)$.
In such situations it is more natural to study the decay of the semigroup
associated with $-\Delta_D^{\Omega_\theta}$ shifted
by the lowest point in its spectrum.
As a particularly interesting situation we mention
the case of periodically twisted tubes,
for which a systematic analysis based on the Floquet-Bloch decomposition
could be developed in the spirit of
\cite{Duro-Zuazua_2000,Ortega-Zuazua_2000}.

We expect that the extra decay rate will be induced
also in other twisted models for which Hardy inequalities
have been established recently \cite{K3,KK2}.

It would be also interesting to study the effect of twisting
in other physical models. As one possible direction
of this research, let us mention the question
of the long time behaviour of the solutions
to the dissipative wave equation
\cite{Gallay-Raugel_1997,Gallay-Raugel_1998,Orive-Pazoto-Zuazua_2001}.

Let us conclude the paper by a general conjecture.
We expect that there is always an improvement
of the decay rate for the heat semigroup if a Hardy inequality holds:
\begin{Conjecture}
Let~$\Omega$ be an open connected subset of~$\Real^d$.
Let~$H$ and $H_+$ be two self-adjoint operators in $\sii(\Omega)$
such that
$
  \inf\sigma(H)
  = \inf\sigma(H_+)
  = 0
$.
Assume that there is a positive smooth function $\varrho:\Omega\to\Real$
such that $H_+ \geq \varrho$,
while $H-V$ is a negative operator
for any non-negative non-trivial $V \in C_0^\infty(\Omega)$.
Then there exists a positive function $K:\Omega\to\Real$
such that
$$
  \lim_{t\to\infty}
  \frac{\|e^{- H_+ t}\|_{\sii(\Omega,K)\to\sii(\Omega)}}
  {\|e^{- H t}\|_{\sii(\Omega,K)\to\sii(\Omega)}} = 0
  \,.
$$
\end{Conjecture}
\noindent
A similar conjecture can be stated for the same type of operators
in different Hilbert spaces.
In this paper we proved the conjecture for the special situation where
$H=H_0-E_1$ and $H_+=H_\theta-E_1$ (transformed Dirichlet Laplacians)
in $\sii(\Omega)$, with $\Omega=\Omega_0$ (unbounded tube).
In general, the proof seems to be a hardly accessible problem.

\section*{Acknowledgment}
The first author would like to thank
the Basque Center for Applied Mathematics in Bilbao
where part of this work was carried out,
for hospitality and support.
The work was partially supported by the Czech Ministry of Education,
Youth and Sports within the project LC06002,
and by Grant MTM2008-03541 of the MICINN (Spain).
%
\addcontentsline{toc}{section}{References}
\bibliography{bib}

\providecommand{\bysame}{\leavevmode\hbox to3em{\hrulefill}\thinspace}
\providecommand{\MR}{\relax\ifhmode\unskip\space\fi MR }
\providecommand{\MRhref}[2]{%
  \href{http://www.ams.org/mathscinet-getitem?mr=#1}{#2}
}
\providecommand{\href}[2]{#2}
\begin{thebibliography}{10}

\bibitem{Brezis_FR}
H.~Br{\'e}zis, \emph{{Analyse fonctionnelle: Th\'eorie et applications}},
  Dunod, 2002.

\bibitem{Briet-Kovarik-Raikov-Soccorsi_2008}
Ph. Briet, H.~Kova\v{r}{\'\i}k, G.~Raikov, and E.~Soccorsi, \emph{Eigenvalue
  asymptotics in a twisted waveguide}, Commun. in Partial Differential
  Equations, to appear; preprint on arXiv:0808.1528v2 [math.SP] (2008).

\bibitem{Cabre-Martel_1999}
X.~Cabr\'e and Y.~Martel, \emph{{Existence versus explosion instantan\'ee pour
  des \'equations de la chaleur lin\'eaires avec potentiel singulier}}, C. R.
  Acad. Sci. Paris \textbf{329} (1999), 973--978.

\bibitem{Daners}
D.~Daners, \emph{Dirichlet problems on varying domains}, J.~Differential
  Equations \textbf{188} (2003), 591--624.

\bibitem{Duro-Zuazua_1999}
G.~Duro and E.~Zuazua, \emph{Large time behavior for convection-diffusion
  equations in {$\mathbb{R}^N$} with asymptotically constant diffusion},
  Commun. in Partial Differential Equations \textbf{24} (1999), 1283--1340.

\bibitem{Duro-Zuazua_2000}
\bysame, \emph{Large time behavior for convection-diffusion equations in
  {$\mathbb{R}^N$} with periodic coefficients}, J. Differential Equations
  \textbf{167} (2000), 275--315.

\bibitem{EKK}
T.~Ekholm, H.~Kova{\v{r}}{\'\i}k, and D.~Krej\v{c}i\v{r}\'{\i}k, \emph{A
  {H}ardy inequality in twisted waveguides}, Arch. Ration. Mech. Anal.
  \textbf{188} (2008), 245--264.

\bibitem{Escobedo-Kavian_1987}
M.~Escobedo and O.~Kavian, \emph{Variational problems related to self-similar
  solutions of the heat equation}, Nonlinear Anal.-Theor. \textbf{11} (1987),
  1103--1133.

\bibitem{EKov_2005}
P.~Exner and H.~Kova\v{r}{\'\i}k, \emph{Spectrum of the {S}chr{\"o}dinger
  operator in a perturbed periodically twisted tube}, Lett. Math. Phys.
  \textbf{73} (2005), 183--192.

\bibitem{Friedlander-Solomyak_2007}
L.~Friedlander and M.~Solomyak, \emph{On the spectrum of the {D}irichlet
  {L}aplacian in a narrow strip}, Israeli Math. J. \textbf{170} (2009), no.~1,
  337--354.

\bibitem{Gallay-Raugel_1997}
Th. Gallay and G.~Raugel, \emph{Stability of travelling waves for a damped
  hyperbolic equation}, Z. angew. Math. Phys. \textbf{48} (1997), 451--477.

\bibitem{Gallay-Raugel_1998}
\bysame, \emph{Scaling variables and asymptotic expansions in damped wave
  equations}, J. Differential Equations \textbf{150} (1998), 42--97.

\bibitem{Griffiths}
D.~J. Griffiths, \emph{Introduction to quantum mechanics}, Prentice Hall, Upper
  Saddle River, NJ, 1995.

\bibitem{Kovarik-Sacchetti_2007}
H.~Kova{\v{r}}{\'\i}k and A.~Sacchetti, \emph{{Resonances in twisted quantum
  waveguides}}, J. Phys. A \textbf{40} (2007), 8371--8384.

\bibitem{KK2}
H.~Kova\v{r}{\'\i}k and D.~Krej\v{c}i\v{r}\'{\i}k, \emph{{A Hardy inequality in
  a twisted Dirichlet-Neumann waveguide}}, Math. Nachr. \textbf{281} (2008),
  1159--1168.

\bibitem{K3}
D.~Krej\v{c}i\v{r}\'{\i}k, \emph{Hardy inequalities in strips on ruled
  surfaces}, J. Inequal. Appl. \textbf{2006} (2006), Article ID 46409, 10
  pages.

\bibitem{K6}
D.~Krej\v{c}i\v{r}\'{\i}k, \emph{Twisting versus bending in quantum
  waveguides}, Analysis on Graphs and its Applications, Cambridge, 2007
  (P.~Exner et~al., ed.), Proc. Sympos. Pure Math., vol.~77, Amer. Math. Soc.,
  Providence, RI, 2008, pp.~617--636.

\bibitem{K6-erratum}
\bysame, \emph{Twisting versus bending in quantum waveguides},
  arXiv:0712.3371v2 [math-ph] (2009), corrected version of~\cite{K6}.

\bibitem{Orive-Pazoto-Zuazua_2001}
R.~Orive, A.~Pazoto, and E.~Zuazua, \emph{Asymptotic expansion of damped wave
  equations with periodic coefficients}, Math. Models Methods Appl. Sci.
  \textbf{11} (2001), 1285--1310.

\bibitem{Ortega-Zuazua_2000}
J.~Ortega and E.~Zuazua, \emph{Large time behavior in {$\mathbb{R}^N$} for
  linear parabolic equations with periodic coefficients}, Asymptotic Anal.
  \textbf{22} (2000), 51--85.

\bibitem{Pinchover_2007}
Y.~Pinchover, \emph{Topics in the theory of positive solutions of second-order
  elliptic and parabolic partial differential equations}, Spectral Theory and
  Mathematical Physics: A Festschrift in Honor of Barry Simon's 60th Birthday
  ({F.~Gesztesy, et al.}, ed.), Proc. Sympos. Pure Math., vol.~76, Amer. Math.
  Soc., Providence, RI, 2007, pp.~329--356.

\bibitem{Vazquez-Zuazua_2000}
J.~L. V{\'a}zquez and E.~Zuazua, \emph{The {H}ardy inequality and the
  asymptotic behaviour of the heat equation with an inverse-square potential},
  J. Funct. Anal. \textbf{173} (2000), 103--153.

\bibitem{Weidmann_1980}
J.~Weidmann, \emph{Continuity of the eigenvalues of self-adjoint operators with
  respect to the strong operator topology}, Integral Equations and Operator
  Theory \textbf{3} (1980), no.~1, 138--142.

\end{thebibliography}
\bibliographystyle{amsplain}
\end{document}